\documentclass[final,hidelinks,onefignum,onetabnum]{siamart220329}
\usepackage{amsmath}
\usepackage{amssymb}
\usepackage{graphicx}
\usepackage{mathtools}
\usepackage{bm}
\usepackage{xcolor}
\usepackage{caption}
\usepackage{subcaption}
\usepackage{algorithmic}
\usepackage{algorithm}
\usepackage{multirow}
\usepackage{makecell}
\usepackage{lipsum}
\usepackage{amsfonts}
\usepackage{graphicx}
\usepackage{epstopdf}
\usepackage{amsopn}
\usepackage{tabularx}

\headers{Randomized Kaczmarz in Adversarial Distributed  Setting}{L. Huang, X. Li, and D. Needell}

\title{ Randomized Kaczmarz in Adversarial Distributed  Setting \thanks{ 
This work is partially supported by AMS Simons Travel grant, NSF  DMS \#2011140 and NSF DMS \#2108479. }}

\author{Longxiu Huang\thanks{Department of Computational Mathematics, Science and Engineering and Department of Mathematics, Michigan State University, MI (\email{huangl3@msu.edu}).}
\and Xia Li\thanks{ Microsoft Corporation, WA (Corresponding author: \email{xli51@g.ucla.edu} or \email{xiali1@microsoft.com}).}
\and Deanna Needell\thanks{Department of  Mathematics, University of California Los Angeles, CA (\email{deanna@math.ucla.edu}).}
}

\ifpdf
\hypersetup{
  pdftitle={Distributed Randomized Kaczmarz for the Adversarial Workers},
  pdfauthor={L. Huang, X. Li, and D. Needell}
}

\ifpdf
  \DeclareGraphicsExtensions{.eps,.pdf,.png,.jpg}
\else
  \DeclareGraphicsExtensions{.eps}
\fi

\newsiamremark{remark}{Remark}
\newsiamremark{hypothesis}{Hypothesis}
\crefname{hypothesis}{Hypothesis}{Hypotheses}
\newsiamthm{claim}{Claim}

\newsiamthm{problem}{Problem} 

\newsiamthm{question}{Question}
\newsiamthm{example}{Example}

\begin{document}

\maketitle

\begin{abstract}
Developing large-scale distributed methods that are robust to the presence of adversarial or corrupted workers is an important part of making such methods practical for real-world problems. In this paper, we propose an iterative approach that is adversary-tolerant for convex optimization problems. By leveraging simple statistics, our method ensures convergence and is capable of adapting to adversarial distributions. Additionally, the efficiency of the proposed methods for solving convex problems is shown in simulations  with the presence of adversaries. Through simulations, we demonstrate the efficiency of our approach in the presence of adversaries and its ability to identify adversarial workers with high accuracy and tolerate varying levels of adversary rates.
\end{abstract}



\section{Introduction}
As machine-learning algorithms gain popularity in industrial applications, it is critical to make them and their optimization subroutines robust and adversary-tolerant. These attacks can take various forms, including evasion \cite{evation}, data poisoning \cite{data_poison} and model extraction \cite{model_privacy, sesame}. In large-scale machine learning problems, which are often run on distributed systems, attacks can come in the form of Byzantine attacks \cite{byz_General}, where individual computing units, also known as `worker machines' or simply `workers', may produce arbitrarily or maliciously results. Real world applications include protecting a circular system against failure of
individual chips, or a ballistic missile defense system using redundant computing sites to protect against the destruction of individual sites by a nuclear attack \cite{lamport2019byzantine}. Especially, in the setting of linear systems, one application is monitoring and control systems for critical infrastructure such as power grids \cite{SWIRYDOWICZ2022102870} and water supply networks \cite{doi:10.1080/00207543.2013.850550}. Here adversarial agents may attempt to disrupt these systems by injecting erroneous data or compromising sensor readings.  

A commonly used approach to mitigate Byzantine attacks is to use \textit{redundancy} \cite{lamport2019byzantine}; that is, to request the same computation from multiple workers.
The main challenge with such an approach is how to leverage the outputs from these workers  efficiently, and in such a way that even seemingly catastrophic adversarial outputs can be identified and tolerated. 
Let's consider the optimization problem of the following form:
\begin{equation}\label{eqn:general_obj}
\min_{x \in \mathbb{R}^{d_2}}F(x)=\sum_{i=1}^{d_1}f_i(x)
\end{equation} 
where $d_1$ is a positive integer. To solve the problem iteratively, we use gradient descent method to update the estimate:   
\begin{equation}\label{eqn:general_update}
x_{j+1} = x_{j} + \gamma_{j}\sum_{i=1}^{d_1}\nabla f_i(x_j) 
\end{equation}
with some step-size $ \gamma_{j}$. Such objective functions lend themselves naturally to distributed algorithms. In the distributed  setting, the central worker distributes $f_i$ among the workers. Each worker returns the corresponding gradient $\nabla f_i(x_j)$ and the central worker aggregates those returns  to compute or approximate the updating step \eqref{eqn:general_update}.  In particular, we illustrate our method on solving an over-determined linear system $Ax=b$. However, the algorithms can be easily adapted for \eqref{eqn:general_obj}. The linear system can be modeled as a least squares problem $\min_{x}\|Ax-b\|_2^2$ and the least squares problem can be rewritten in the form of \eqref{eqn:general_obj} with $f_i(x_j) = \frac{1}{2}(A_ix_j - b_i)^2$, where $A \in \mathbb{R}^{d_1\times d_2}, b\in \mathbb{R}^{d_1} $, $A_i$ is the $i$-th row of $A$, and $b_i$ is the $i$-th component of $b$. 
The central worker partitions the data matrix $A$ into rows $A_i$ and the rows are distributed among the workers. In the linear setting, each worker only needs to return the scalar  $A_ix_j - b_i$ instead of  the gradient $(A_ix_j - b_i)A_i^\top$. Then the central worker aggregates those returns and approximate the updates in  \eqref{eqn:general_update}.

In this work, we consider the setting where some of the workers are adversarial, i.e., the workers return noisy results or enormously large results. Our goal is to develop a variant of the randomized Kaczmarz (RK) method \cite{strohmer2009randomized} for adversarial workers to solve the linear system $Ax=b$. For readers' convenience, we restate the RK method in Alg.~\ref{alg:kac_alg1}.
\begin{algorithm}[!h]
\caption{Randomized Kaczmarz Algorithm}\label{alg:kac_alg1}
\begin{algorithmic}[1]
\STATE Select a row index $i_j \in [d_1]$ with probability $\frac{\|A_{i_j}\|_2^2}{\|A\|_F^2}$
    \STATE Update $x_{j+1} = \arg\min_{x\in \mathbb{R}^{d_2}}\|x-x_j\|$  s.t. $A_{i_j}x_{j+1} = b_{i_j}$
    \STATE Repeat until convergence
    \end{algorithmic}
\end{algorithm}
We assume that there is one central worker $w_c$ and $N$ workers in total, among which $p$ fraction of the unknown workers are adversarial and there are $k$ error categories in total. During the initial data distribution, each row $A_r$ is distributed to $N_r$ workers. Among those $N_r$ workers, workers in the $\ell$-th category $C_{\ell}$ consist of $p_{r,\ell}$ fraction of all workers. We assume $C_0$ contains all reliable workers. The total adversarial rate for row $r$ is $p_r = \sum_{\ell=1}^{k}p_{r,\ell}$ and the fraction of reliable workers for row $r$ is $p_{r,0}$. We assume $p_{r,\ell} < 1- p_r = p_{r,0}$, for all $r$, and $\ell\neq 0$. {In particular, we assume that an adversarial worker $w_s^r$ in category $C_{\ell}$ returns the residual $c_{s}^r=b_r+e_{\ell,r}-\langle x_{j},A_r^T \rangle$, $e_{\ell,r} \in \mathbb{R}$, and a reliable worker returns $c_{s}^r=b_r -\langle x_{j},A_r^T \rangle$.} 
Our approach utilizes simple statistics to identify and ignore adversarial results, and thus the setting in which the adversarial workers communicate and select among 
$k$ types of errors to output is the most challenging for our approach. 

\subsection{Contribution}
Our key contributions are threefold:
(i) develop efficient methods and algorithms to guarantee accurate estimates for the true solution in the presence of adversarial workers; (ii) identify the adversarial workers efficiently; (iii) provide theoretical convergence analysis for solving linear systems with a portion of workers being adversarial.
\subsection{Related work}
\noindent\textbf{Kaczmarz method.~}
The Kaczmarz method is an iterative technique for solving linear systems that was first introduced in   1937   by Kaczmarz \cite{karczmarz1937angenaherte}. In computer tomography, the method is also referred to as the Algebraic Reconstruction Technique (ART) \cite{gordon1975image,herman1993algebraic,natterer2001mathematics}. The method has a broad range of applications, from computer tomography to digital signal processing. Later Strohmer et al.  proposed a randomized version of the Kaczmarz method (RK) \cite{strohmer2009randomized} in the context of consistent linear systems.   They proved that RK has an exponential bound on the expected rate of convergence, with the probability of selecting each row proportional to the squared Euclidean norm of that row.  The method has also been adapted to handle inconsistent linear systems \cite{petra2016single,popa1999characterization,bai2021,ma2015convergence}. For example, Needell proved in \cite{Needell2009RandomizedKS} that RK converges for inconsistent linear systems to a horizon that depends on the size of the largest entry of the noise. An adaptive maximum-residual sampling strategy has also been analyzed for the inconsistent extension \cite{petra2016single}. Additionally, RK has been studied in the context of solving systems of linear inequalities \cite{leventhal2010randomized,agmon_1954_mm,bai2019partially} as well as systems of tensor equations \cite{wang2023solving}.


\noindent\textbf{Robust optimization.~}
In optimization problems, practical challenges often arise due to various factors such as errors in data collection and transmission, adversarial or non-responsive workers (also known as stragglers), and corruptions in modern storage systems. In the setting of games, adversarial workers are agents that can directly influence the decision-making rules of a subset of agents and change the dynamics of the games which can eventually degrade the performance of the system \cite{borowski2015understanding}. Similarly, in the network setting, agents who may wish to collaborate as adversaries can influence the equilibrium or consensus \cite{vyavahare2019distributed}. In machine learning, adversarial workers send gradients that include carefully crafted malicious patterns which are designed to introduce biases or vulnerabilities into the global model when aggregated (model poisoning attack). Other adversaries include gradient flipping \cite{6868233} (adversarial nodes flip the sign of gradients or  modify them to push the global model in the wrong direction) and backdoor attack \cite{gao2020backdoor} (adversarial workers inject backdoor patterns into their local gradients).

To address these challenges, researchers have proposed various mitigation strategies.  
For instance, to tackle the issue of straggling workers, several encoding schemes have been proposed in literature. For example, Gordon et al. \cite{GORDON1970471} and Karakus et al. \cite{karakus2017straggler} introduced methods to embed redundancy directly in the data, while Bitar et al. \cite{bitar2020stochastic} proposed a gradient-coding scheme for straggler mitigation when stragglers are uniformly random. 

Another important branch in the analysis of SGD-type methods is to deal with robustness to adversaries from the data. Chi et al. \cite{Chi2019MedianTruncatedGD} and Haddock et al. \cite{haddock2020quantile} designed quantile-based methods to solve corrupted linear equations. Yang et al. \cite{yang2019byrdie} proposed a variant of the gradient descent method based on the geometric median to deal with adversarial workers, while Alistarh et al. \cite{alistarh2018byzantine} discussed the problem of stochastic optimization in an adversarial setting where the workers sample data from a distribution and an $\alpha$ fraction of them may adversarially return any vector. However, these methods are limited to scenarios where the adversary rate is less than $\frac{1}{2}$. Our proposed algorithm, on the other hand, can converge to the exact solution even with an adversary rate higher than $\frac{1}{2}$ by utilizing redundancy.

\section{Method}
In this section, we present a simple and efficient mode-based method for solving linear systems in the presence of adversarial workers, as well as identifying potential adversarial workers which may be placed in a block-list (more details will be provided later). The method detects the mode category based on the size of the returned result groups. More specifically, for each row, the central worker groups similar results and selects the result from the group with the largest size, referred to as the {\bf mode}. From these modes across all selected rows, the central worker then updates the guess with the mode with the largest size. If there is only one row, the central worker updates the guess with the mode.

Given the number of used workers $n_r$ for a specific row $r$, the expected number of workers from category $C_{\ell}$ is ${n_rp_{r,\ell}}$ \footnote{The central worker determines the number of different result groups during the first $m$ iterations and takes the maximum number} and the number of non-adversarial workers is $n_r(1-p_r)$ with $p_r=\sum_{\ell=1}^{k}p_{r,\ell}$. In practice, a group with the maximum size is randomly selected and used to update the guess as long as its size is greater than $n_r(1-p_r)$ (as shown in Alg.~\ref{alg:multi_block_list} Line~\ref{alg:line:9}). If the algorithm is implemented with a block-list, the block-list is updated through a frequency-based approach throughout the iterations: each row has a counter that records whether a worker is selected but fails to be the mode during each iteration. For every updating cycle $S$, the worker with the largest count in each counter is identified as a potential adversarial worker and placed in the block-list (as shown in  Alg.~\ref{alg:multi_block_list} ~Line~\ref{alg:line:13} -- \ref{alg:line:15}). Once a worker is in the block-list, it will not be considered in future iterations. The full details of the algorithm can be found in Alg.~\ref{alg:multi_block_list} and the related theoretical results are provided in the following section.

\begin{algorithm}[!th]
\caption{{Distributed randomized Kaczmarz with/without block-list}}\label{alg:multi_block_list}
\begin{algorithmic}[1]
    \STATE \textbf{Input}: Worker sets $D_r$, the number of used worker for each row $n_r$ the number of used rows $d_0$, $\text{MaxIter}$, $\text{Tol}$, updating cycle $S$, Blocklist\_flag
    \STATE Initialize $c_s = 2 \cdot\text{Tol}$, a counter vector $E_r = 0\in \mathbb{R}^{N_r}$ for each row $r$
    \IF{Blocklist\_flag}
       \STATE Initialize \textbf{block-list $B$}  
    \ENDIF
    \WHILE{$j < \text{MaxIter}$ and $|c_s| > \text{Tol}$, }
    \STATE The central worker $w_c$ selects rows $\tau$  from $[d_1]$ uniformly  at random
    \STATE Sample $w_1^r,\ldots,w_{n_r}^r$ for each row $r \in \tau$, uniformly at random from $D_r$
    \STATE Broadcast $A_{r}$ to $w_1^r,\ldots,w_{n_r}^r$ 
    \STATE  
    $w_s^r$ returns $c_s^r = \frac{\langle A_{r},x_i\rangle-(b_{r}+ e_{\ell,r})}{\|A_{r}\|^2} $, if  $w_s\in C_{\ell}$ 
    \FOR{$r \in \tau$
    }
    \STATE $w_c$ splits  $\{c_s^r\}_{s = 1}^{n_r}$ into groups $G_1,\ldots,G_{k_r}$
    \STATE {Set $G_{s^*}^r=G_{s^*}$}, where $s^*=\arg\max\limits_{s\in S}|G_{s}|$ and $S=\{s:|G_{s}| \geq n_r(1-\sum_{l=1}^{k_r}p_{r,l})\}$\label{alg:line:9}
     \STATE Update $E_r(s) = E_r(s) + 1,$ if $c_s^r\notin G_{s^*}^r$\footnotemark\label{alg:line:13}
    \ENDFOR
    \STATE $c_{s^*} = \max_{r\in \tau} |c_{s^*}^{r}|$ {and $i_j = \arg\max_{r\in \tau} |c_{s^*}^{r}|$,} where $c_{s^*}^{r} \in G_{s^*}^{r}$

   \STATE Update $x^{j+1} = x^{j} + c_{s^*}A_{i_j}^\top$
  
   \IF {Blocklist\_flag \& mod$(j,S) = 0$}\label{alg:line:10}
   \FOR{$r \in \tau$
    }
        \STATE Update $B$ by checking the value of entries in $E_r$\label{alg:line:15}
        \STATE $D_r = D_r\setminus B$
    \ENDFOR
   \ENDIF
   \STATE Update $j=j+1$
    \ENDWHILE
    \IF{Blocklist\_flag}
        \STATE \textbf{Output}: $x^j$ and $B$
    \ELSE
    \STATE \textbf{Output}: $x^j$
    \ENDIF
    \end{algorithmic}
\end{algorithm}
\footnotetext{Note for a $n_r$-dimensional vector $E_r$, we adopt $E_r(i)$ to denote the $i$-th entry of the vector.}

\section{Theoretical results}
In this section,  we   provide a rigorous theoretical analysis of the mode distributions and convergence behavior of our mode-based method. To simplify the presentation, 
we  provided a summary of the key notation used in our analysis in Table~\ref{tab:notation}.
         

\begin{table}[h]
\caption{Notation Table}
    \label{tab:notation}
    \centering
    \begin{tabular}{|l|l|}
      \hline
        $A$ & Data matrix $A$, $A \in \mathbb{R}^{{d_1}\times{d_2}}$ \\
         \hline
        $\Tilde{A}$ & Row normalized version of matrix $A$\\
        \hline
        $N$ & Number of workers in total \\
        \hline
        $N_r$ & Number of workers holding row $r$ \\
        \hline
        $n_r$ & Number of workers chosen for row $r$ \\
        \hline
         $C_{\ell}$ & $\ell$-th error category\\
        \hline
        $k$ & Number of error categories in total\\
        \hline
        $e_{r,\ell}$ & Error of the $\ell$-th error category for a row $r$\\
        \hline
        $e_{r}$ & \makecell[l]{Vector form of errors in all error categories of row $r$, \\
        $e_{r} \coloneqq (e_{r,1},\ldots,e_{r,k})^{\top}$}\\
        \hline
        $e$ & \makecell[l]{Matrix form of errors in all error categories of all rows,\\ $e\coloneqq(e_{r,\ell})_{r,\ell}$}\\
        \hline
        $d_0$ & Number of chosen rows  per iteration \\
        \hline
        $p_{r,\ell}$ & Fraction of workers holding row $r$ in category $\ell$\\ 
        \hline
        $\hat{q}_{\text{mode}}^{\ell,r}$ & \makecell[l]{Probability that there is a mode among the outputs\\
        of chosen workers of row $r$ and the mode is in the category $\ell$ \\(see Lemma~\ref{lem: q})}\\
         \hline
         $q^r$ & \makecell[l]{Probability that there is a mode among the outputs \\
         of chosen workers for row $r$ (see Lemma~\ref{lem: q})}\\
         \hline
         $[d_1]$ & Set of the integers from 1 to $d_1$, $[d_1]\coloneqq\{1,\ldots,d_1\}$\\
         \hline
         $\binom{[d_1]}{d_0}$ &  Collection of the subsets of $[d_1]$ with $d_0$ elements \\
         \hline
         $\text{unif}\left(\binom{[d_1]}{d_0}\right)$ &   Uniform sampling from the collection $\binom{[d_1]}{d_0}$.\\
         \hline
         $\tau_i$ & Index set of chosen rows at $i$-th iteration, $|\tau_i| = d_0$\\
         \hline
         $\tau_i'$ & Index set of chosen rows that have a mode, $\tau_i' \subset \tau_i$\\
         \hline
         $t_i$ &   Row  index that has the largest mode number, $t_i = t(x_{i-1},\tau_i)$\\
         \hline
         $\mathbb{P}(r \text{ mode}, g, \ell) $ & \makecell[l]{Probability that the mode is in the category $\ell$\\
         with mode number $g$ for row $r$ (see Lemma~\ref{lem:mode g})}\\
         \hline
         $\mathbb{P}(t_i,\ell,g|\tau_i,x_{i-1})$ & \makecell[l]{Probability that a mode 
         is  from  row $t_i$ among rows $\tau_i$  
         and \\the mode is in  category $C_{\ell}$ with  mode number $g$,    given \\
     the previous estimate $x_{i-1}$. It is also denoted by $\mathbb{P}(t_i,\ell,g)$,}\\
         \hline
         $\mathbb{P}(t_i,g)$ & \makecell[l]{Probability that the mode is from row $t_i$ 
         among rows $\tau_i$ \\ with mode number $g$ 
         provided the previous  estimate $x_{i-1}$,\\ more details refer to Corollary~\ref{cor:pro_rowt_i}.
         } \\
         \hline
    \end{tabular}
\end{table}
 
\subsection{Mode distribution}
Algorithm~\ref{alg:multi_block_list} utilizes the mode to identify adversarial workers and achieve convergence. In this section, we discuss the calculation of the probability of a specific category $\ell$ being the mode of a given row $r$ during each iteration of the algorithm.   For simplicity, let $C_0$ denote the category of ``reliable" workers (workers return correct results). For each row $r$, the fraction of reliable workers holding row $r$ is $p_{r,0}=1-p_r$. 
 We use $d_0$ rows for the computation per iteration. Recall that each row $r$ is held by $N_r$ workers (fixed). Among those $N_r$ workers, workers in the category $\ell$ take up a fraction of $p_{r,\ell}$. At each iteration, the central worker chooses a set of row indices of size $d_0$ uniformly at random  and requests the corresponding workers to return their results. More specifically, given a set of row indices $\tau_i$ at $i$-th iteration, the central worker first finds the modes among the results from each row $r \in \tau_i$ and among those modes, chooses the mode with the largest group size (``the majority vote").  
 
For any row $r$, let $a^r_{g,\ell}$ be the coefficient of the monomial  $x^{n_r-g}$ of the polynomial  
\begin{eqnarray*}
\prod\limits_{{\ell'}=0,{\ell'}\neq \ell}^{k}\left(\sum_{j=0}^{g-1}\binom{N_rp_{r,{\ell'}}}{j}x^j\right).
\end{eqnarray*} 
Let $b_{g}^{r}$ be the coefficient of the term $x^{n_r}$ of the polynomial
\[
\prod_{\ell=0}^k\left( \sum_{j=0}^{g-1}\binom{N_rp_{r,\ell}}{j}x^j\right)
\]
\begin{lemma}\label{lem:mode g}
For   row $r$, the probability that the mode is in the category $\ell$ with mode number $g$ is
$
\mathbb{P}(r \text{ mode}, g, \ell) = \frac{\binom{N_r p_{r,\ell}}{g}a^r_{g,\ell}}{\binom{N_r}{n_r}}$.
\end{lemma}
\begin{proof}
    See Appendix \ref{pf:lem_mode_g}.
\end{proof}
Using Lemma~\ref{lem:mode g}, we obtain the following conclusions by going over all possible mode numbers and  all error categories.
\begin{lemma}\label{lem: q}
For row $r$, the probability that the category $\ell$ is the mode is 
\[
 \hat{q}_{mode}^{{\ell},r} 
 =\sum_{g = g_0(r)}^{n_r}\frac{\binom{N_r p_{r,\ell}}{g}a_{g,\ell}^r}{\binom{N_r}{n_r}},
\]
where 
$g_0(r) = \max(\lceil\frac{n_r}{k+1}\rceil,  \lceil n_rp_{r,0}\rceil)$, and the probability that  there is a mode with mode number $g$ for the calculation of row $r$ 
is
\[
 q_{g}^{r} 
 =\sum_{\ell=0}^{k}\frac{\binom{N_r p_{r,\ell}}{g}a_{g,\ell}^r}{\binom{N_r}{n_r}}.
\] 
Additionally,  the probability that there is a mode for the calculation of row $r$ is
\begin{equation}
q^r = \sum_{\ell = 0}^{k}\hat{q}_{mode}^{{\ell},r} 
=\sum_{g =g_0(r)}^{n_r}\sum_{\ell=0}^{k}\frac{\binom{N_r p_{r,\ell}}{g}a_{g,\ell}^r}{\binom{N_r}{n_r}},
\end{equation}
where $\binom{n}{g} = 0$ if $n < g$, for any integer $n$. 
\end{lemma}
 
In the following lemma, we also calculate the probability $\mathbb{P}(t_i, \ell, g|\tau_i, x_{i-1})$ that a mode is from row $t_i$ in the category $C_{\ell}$ with a mode number $g$ when rows $\tau_{i}$ are used in the computation and the previous estimate $x_{i-1}$ is given. For simplicity, we omit the condition of $\tau_i,x_{i-1}$ in the notation and denote $\mathbb{P}(t,\ell,g|\tau_i,x_{i-1})$ by $\mathbb{P}(t_i,\ell,g)$.

\begin{lemma}
Given  the previous estimate $x_{i-1}$ and  row indices $\tau_i$, we have
\[
\mathbb{P}(t_i,\ell,g) = \frac{\binom{N_{t_i}p_{t_i,\ell}}{g}a^{t_i}_{g,\ell}}{\binom{N_{t_i}}{n_{t_i}}}\prod_{s\in\tau_i\setminus{t_i}}\frac{b^s_g}{\binom{N_s}{n_s}}.\]
\end{lemma}
\begin{proof}
The probability that the mode  is produced from category $\ell$ of  row $t_i$ with  the mode number   $g$ can be expressed as
\begin{align*}
  \mathbb{P}(t_i,\ell,g)
    = & \mathbb{P}(t_i \text{ mode}, g, \ell)\times\mathbb{P}(\text{ the largest mode number from }t_i|t_i\text{ mode},\ell,g)\\
    = & \frac{\binom{N_{t_i}p_{t_i,\ell}}{g}a^{t_i}_{g,\ell}}{\binom{N_{t_i}}{n_{t_i}}}\times\mathbb{P}(\text{ the largest mode number from } t_i|t_i\text{ mode},\ell,g)\\
    = & \frac{\binom{N_{t_i}p_{t_i,\ell}}{g}a^{t_i}_{g,\ell}}{\binom{N_{t_i}}{n_{t_i}}}\prod_{s\in\tau_i\setminus{t_i}}\frac{b^s_g}{\binom{N_s}{n_s}}.   
\end{align*}

\end{proof}
Taking the modes produced from different categories into account, we can easily obtain the following result.
\begin{corollary}\label{cor:pro_rowt_i}
Given the previous estimate $x_{i-1}$ and the row indices $\tau_i$, 
the probability that a mode is from row $t_i$ with mode number $g$ 
is 
\begin{equation}
    \begin{aligned}
        \mathbb{P}(t_i,g)
        =& \sum_{\ell = 0}^k\mathbb{P}(t_i,\ell,g) 
        = \sum_{\ell = 0}^k \frac{\binom{N_{t}p_{t_i,\ell}}{g}a^{t_i}_{g,\ell}}{\binom{N_{t_i}}{n_{t_i}}}\prod_{s\in\tau_i\setminus\{t_i\}}\frac{b^s_g}{\binom{N_s}{n_s}}
        = q^{t_i}_g\prod_{s\in\tau_i\setminus{t_i}}\frac{b^s_g}{\binom{N_s}{n_s}}.
    \end{aligned}
\end{equation}
\end{corollary}

\subsection{Convergence without block-list}
 In this section,  our main goal is to provide theoretical error bound for the method without block-list (i.e., Alg.~\ref{alg:multi_block_list} without block-list). The main result for this section is present below. \begin{theorem}\label{thm:multiple-main}
Let $A\in\mathbb{R}^{d_1\times d_2}$ with $d_1\geq d_2$ and $b,e_1,\ldots,e_{k}\in\mathbb{R}^{d_1}$. Assume that we solve $Ax^*=b$ via Alg.~\ref{alg:multi_block_list} without block-list; then 
\begin{equation}
    \begin{aligned}
      \mathbb{E}\|x_{i}-x^*\|_2^2
      \leq &\alpha^{i}\|x_0-x^*\|_2^2+\frac{1-\alpha^{i+1}}{1-\alpha}\sum_{t\in[d_1]}\beta_t\|\widetilde{e}_t\|_2^2\, , 
    \end{aligned}
    \label{eqn:convg2}
\end{equation}
where 
\begin{equation}\label{eq:mul_main_para}
    \begin{aligned}
    &\alpha=1-Q_{\min}\frac{d_0}{d_1}\sigma_{\min}^2(\tilde{A}),\\
    &Q_{\min} =\min_{g, t_i,\tau_i}\sum_{g=g_0({t_i})}^{n_t}q^{t_i}_g\prod_{s\in\tau_i\setminus\{t_i\}}\frac{b^s_g}{\binom{N_s}{n_s}},\\
    &Q_{\max}(t,g,\tau_i)=\max_{\ell} \frac{\binom{N_tp_{t,\ell}}{g}a^t_{g,\ell}}{\binom{N_t}{n_t}}\prod_{s\in\tau_i}\frac{b^s_g}{\binom{N_s}{n_s}},\\
    &\|\widetilde{e}_t\|_2^2 = \sum_{\ell = 0}^k \Tilde{e}_{t,\ell}^2, ~\Tilde{e}_{t,\ell}^2 = \frac{e_{t,\ell}^2}{\|A_{t}\|_2^2}, \\
    &\beta_t = \sum_{\Tilde{\tau}_{t,i}\in\binom{[d_1-1]}{d_0-1}}\sum_{g=g_0({t})}^{n_{t_i}} \frac{1}{\binom{d_1}{d_0}}Q_{\max}(t,g,\tilde{\tau}_{t,i}),
    \end{aligned}
\end{equation}
and $\Tilde{A}$ is the row normalized matrix of  $A$ and $\sigma_{\min}(\tilde{A})$ is $\Tilde{A}$'s  smallest singular value.
\end{theorem}
Before we   prove Theorem~\ref{thm:multiple-main}, we let $t_i$  be  the row  selected at $i$-th iteration to update the guess $x$ and consider solving
\begin{equation}
    \begin{aligned}
        A_{t_i}x&=b_{t_i}+e_{t_i,0} \text{~with~}e_{t_i,0}=0 \\
        A_{t_i}x&=b_{t_i}+e_{t_i,1},\\
        ~&\vdots\nonumber\\
        A_{t_i}x&=b_{t_i}+e_{t_i,k},
    \end{aligned}
\end{equation}
according to some probability distribution.  
Thus, we    have the iteration 
$$
x_{i}=x_{i-1}-\frac{\langle A_{t_i},x_{i-1}\rangle-(b_{t_i}+e_{t_i,\ell})}{\|A_{t_i}\|^2} A_{t_i}^{\top},
$$ 
at $i$-th iteration, where  $\ell\in\{0,1,\ldots,k\}$, and $A_{t_i}$ is the $t_i$-th row of matrix $A$.

In the following analysis, let $\mathbb{E}_{\tau_{i}}$ denote the expectation with respect to the uniformly random sample ${\tau_{i}}$ conditioned upon
{the sampled $\tau_{j}$ for $j < i$, and let $\mathbb{E}$  denote expectation with respect to all random samples ${\tau_{j}}$ for $1 \leq j \leq i$,  where $i $ is the last iteration in the context in which $\mathbb{E}$ is applied.} 

We start our analysis by decomposing the squared error
\begin{align*}
\|x_{i}-x^*\|_2^2
   &=\left\|x_{i-1}-\frac{\langle A_{t_i}^T,x_{i-1}\rangle-(b_{t_i}+e_{t_i,\ell_{t_i}})}{\|A_{t_i}\|_2^2}A^{\top}_{t_i}-x^*\right\|_2^2\nonumber\\
    &=\|x_{i-1}-x^*\|_2^2+\frac{(\langle A_{t_i}^T,x_{i-1}-x^*\rangle-e_{t_i,\ell_{t_i}})^2}{\|A_{t_i}\|^2_2}\nonumber\\
    &~ -\frac{2}{\|A_{t_i}\|^2_2}\langle x_{i-1}-x^*,A_{t_i}^T\rangle(\langle A_{t_i}^T,x_{i-1}-x^*\rangle-e_{t_i,\ell_{t_i}})\nonumber\\
   &=\|x_{i-1}-x^*\|_2^2-\frac{\langle A_{t_i}^T,x_{i-1}-x^*\rangle^2}{\|A_{t_i}\|_2^2}+\frac{e_{t_i,\ell_{t_i}}^2}{\|A_{t_i}\|_2^2}
\end{align*}
Taking the expectation of the above equation, we can easily achieve   that
\begin{equation*}\label{eqn:decomp}
\begin{aligned}
    \mathbb{E}\|x_{i}-x^*\|_2^2 = \|x_{i-1}-x^*\|^2 + \mathbb{E}_{\tau_{i}}\mathbb{E}_{t_i}\mathbb{E}_{\ell_{t_i}}\frac{e_{t_i,\ell_{t_i}}^2}{\|A_{t_i}\|_2^2} - \mathbb{E}_{\tau_{i}}\mathbb{E}_{t_i}\frac{\langle A_{t_i}^T,x_{i-1}-x^*\rangle^2}{\|A_{t_i}\|_2^2}.
\end{aligned}
\end{equation*}
Therefore, the proof of Theorem~\ref{thm:multiple-main} can be divided into the computations of  the conditional expectation of the squared error  from the adversarial workers $\mathbb{E}_{\tau_{i}}\mathbb{E}_{t_i}\mathbb{E}_{\ell_{t_i}}\frac{e_{t_i,\ell_{t_i}}^2}{\|A_{t_i}\|_2^2}$ and the residual part $\mathbb{E}_{\tau_{i}}\mathbb{E}_{t_i}\frac{\langle A_{t_i}^T,x_{i-1}-x^*\rangle^2}{\|A_{t_i}\|_2^2} $ separately which are provided  in the following lemmas.
\begin{lemma}\label{lem:ax-b}
The conditional expectation of squared residual can be bounded below:
\begin{equation}\label{eqn:ax-b}
\mathbb{E}_{\tau_{i}}\mathbb{E}_{t_i}\frac{\langle A_{t_i}^T,x_{i-1}-x^*\rangle^2}{\|A_{t_i}\|_2^2} 
    \geq Q_{\min}\sigma_{\min}^2(\tilde{A})\frac{d_0}{d_1}\|x_{i-1}-x^*\|^2,
\end{equation}
where $\tilde{A}$ is the row normalized version of $A$. Thus, we have
\begin{equation}\label{eqn:ax-b-col}
   \|x_{i-1}-x^*\|^2 - \mathbb{E}_{\tau_{i}}\mathbb{E}_{t_i}\frac{\langle A_{t_i}^T,x_{i-1}-x^*\rangle^2}{\|A_{t_i}\|_2^2}
    \leq (1-Q_{\min}\sigma_{\min}^2(\tilde{A}) \frac{d_0}{d_1})\|x_{i-1}-x^*\|^2 \, ,
\end{equation}
where 
$ 
 Q_{\min} = \min_{g, t_i,\tau_i}\sum_{g=g_0({t_i})}^{n_t}q^{t_i}_g\prod_{s\in\tau_i\setminus\{t_i\}}\frac{b^s_g}{\binom{N_s}{n_s}}.
$
\end{lemma}
\begin{proof}
The expectation of the squared residual can be represented as
    \begin{align*}  &~\mathbb{E}_{\tau_{i}}\mathbb{E}_{t_i}\frac{\langle A_{t_i}^T,x_{i-1}-x^*\rangle^2}{\|A_{t_i}\|_2^2}\\
    = &\mathbb{E}_{\tau_{i}}\sum_{t_i\in\tau_i}\sum_{g=g_0(t_i)}^{n_{t_i}}\sum_{\ell = 0}^k\mathbb{P}(t_i,\ell,g)\left\|\frac{A_{t_i}(x_{i-1}-x^*)}{\|A_{t_i}\|} \right\|^2\\
    = &\sum_{\tau_i\in\binom{[d_1]}{d_0}}p_{x_{i-1}(\tau_i)}\sum_{t_i\in\tau_i}\sum_{g=g_0(t_i)}^{n_{t_i}}\mathbb{P}(t_i,g)\left\|\frac{A_{t_i}(x_{i-1}-x^*)}{\|A_{t_i}\|} \right\|^2\\
    = &\sum_{\tau_i\in\binom{[d_1]}{d_0}}p_{x_{i-1}(\tau_i)}\sum_{t_i\in\tau_i}\sum_{g=g_0(t_i)}^{n_{t_i}}q^{t_i}_g\prod_{s\in\tau_i\setminus\{t_i\}}\frac{b^s_g}{\binom{N_s}{n_s}}\left\|\frac{A_{t_i}(x_{i-1}-x^*)}{\|A_{t_i}\|} \right\|^2.
 \end{align*}
Recall  that $\tau_i \sim \text{unif}(\binom{[d_1]}{d_0}) $. We thus have  $ p_{x_{i-1}}(\tau_i) =1/\binom{d_1}{d_0} = \frac{d_0!(d_1-d_0)!}{d_1!}$. 
 Therefore, 
\begin{align*} \mathbb{E}_{\tau_{i}}\mathbb{E}_{t_i}\frac{\langle A_{t_i}^T,x_{i-1}-x^*\rangle^2}{\|A_{t_i}\|_2^2}
    &\geq Q_{\min}\cdot\frac{d_0!(d_1-d_0)!}{d_1!}\cdot \sum_{\tau_i\in\binom{[d_1]}{d_0}}\sum_{{t_i}\in\tau_i}\left\|\frac{A_{t_i}(x_{i-1}-x^*)}{\|A_{t_i}\|} \right\|^2\\
    & \geq Q_{\min}\cdot\frac{d_0!(d_1-d_0)!}{d_1!}\cdot \binom{d_1-1}{d_0-1}\cdot\|\Tilde{A}(x_{i-1}-x^*)\|^2_{2}\\
    & \geq  Q_{\min}\cdot\frac{d_0}{d_1}\cdot\sigma_{\min}^2(\Tilde{A})\|x_{i-1}-x^*\|^2_{2},
\end{align*}
i.e., \eqref{eqn:ax-b} is derived.
Hence, we also have \eqref{eqn:ax-b-col}. 
\end{proof}

\begin{lemma}\label{lem:err}
The expectation of the squared error from the adversarial workers can be bounded above by $\sum\limits_{t_i\in[d_1]}\beta_{t_i} \|\tilde{e}_{t_i}\|_2^2$, i.e.,
\begin{equation}
   \label{eqn:err}
\mathbb{E}_{\tau_{i}}\mathbb{E}_{t_i}\mathbb{E}_{\ell}\frac{e_{t_i,\ell}^2}{\|A_{t_i}\|_2^2}\leq\sum_{t_i\in[d_1]}\beta_{t_i} \|\tilde{e}_{t_i}\|_2^2,
\end{equation}
where  $\Tilde{e}_{t_i}^2 =\frac{e_{t_i,\ell}^2}{\|A_{t_i}\|_2^2}$ and $\beta_{t_i} = \sum\limits_{\Tilde{\tau}_{t_i,i}\in\binom{[d_1-1]}{d_0-1}}\sum\limits_{g=g_0({t_i})}^{n_{t_i}} \frac{1}{\binom{d_1}{d_0}}Q_{\max}(t_i,g,\tilde{\tau}_{t_i,i})$ with  \\
$Q_{\max}(t_i,g,\tau_i)=\max\limits_{\ell} \frac{\binom{N_{t_i}p_{t_i,\ell}}{g}a^{t_i}_{g,\ell}}{\binom{N_{t_i}}{n_{t_i}}}\prod\limits_{s\in\tau_i}\frac{b^s_g}{\binom{N_s}{n_s}}$.
\end{lemma}
\begin{proof}
The expectation of the squared error from the adversarial workers can be represented as
$\mathbb{E}_{\tau_{i}}\mathbb{E}_{t_i}\mathbb{E}_{\ell}\frac{e_{t_i,\ell}^2}{\|A_{t_i}\|_2^2}=\mathbb{E}_{\tau_{i}}\sum\limits_{{t_i}\in\tau_i}\sum\limits_{g=g_0({t_i})}^{n_{t_i}}\sum\limits_{\ell = 0}^k\frac{\binom{N_{t_i}p_{t_i,\ell}}{g}a^{t_i}_{g,\ell}}{\binom{N_{t_i}}{n_{t_i}}}\prod\limits_{s\in\tau_i\setminus\{{t_i}\}}\frac{b^s_g}{\binom{N_s}{n_s}}\frac{e_{{t_i},\ell}^2}{\|A_{{t_i}}\|_2^2}$.
To simplify the expressions, we let   \\
$Q_{\max}({t_i},g,\tau_i\setminus\{{t_i}\})=\max\limits_{\ell} \frac{\binom{N_tp_{{t_i},\ell}}{g}a^{t_i}_{g,\ell}}{\binom{N_{t_i}}{n_{t_i}}}\prod\limits_{s\in\tau_i\setminus\{{t_i}\}}\frac{b^s_g}{\binom{N_s}{n_s}}$, \\
$\Tilde{e}_{{t_i},\ell}^2 = \frac{e_{{t_i},\ell}^2}{\|A_{{t_i}}\|_2^2}$, 
 and $\Tilde{e}_{{t_i}} = (e_{t,0},e_{{t_i},1},\ldots, e_{{t_i},k})$ . 
Then \eqref{eqn:err} can be achieved by 
\begin{equation*} 
    \begin{aligned}
\mathbb{E}_{\tau_{i}}\mathbb{E}_{t_i}\mathbb{E}_{\ell}\frac{e_{t_i,\ell}^2}{\|A_{t_i}\|_2^2}
    =& \mathbb{E}_{\tau_{i}}\sum_{{t_i}\in\tau_i}\sum_{g=g_0({t_i})}^{n_{t_i}}\sum_{\ell = 0}^k\frac{\binom{N_{t_i}p_{{t_i},\ell}}{g}a^{t_i}_{g,\ell}}{\binom{N_{t_i}}{n_{t_i}}}\prod_{s\in\tau_i\setminus\{t_i\}}\frac{b^s_g}{\binom{N_s}{n_s}}\Tilde{e}_{t_i,\ell}^2\\
    \leq & \mathbb{E}_{\tau_{i}}\sum_{{t_i}\in\tau_i}\sum_{g=g_0({t_i})}^{n_{t_i}}Q_{\max}(t_i,g,\tau_i\setminus\{t_i\})\sum_{\ell = 0}^k \Tilde{e}_{t_i,\ell}^2\\
=&\mathbb{E}_{\tau_{i}}\sum_{{t_i}\in\tau_i}\sum_{g=g_0({t_i})}^{n_{t_i}}Q_{\max}(t_i,g,\tau_i\setminus\{t_i\})\|\Tilde{e}_{t_i}\|_2^2\\
    =&\sum_{\tau_i\in\binom{[d_1]}{d_0}}\frac{1}{\binom{d_1}{d_0}}\sum_{{t_i}\in\tau_i}\sum_{g=g_0({t_i})}^{n_{t_i}}Q_{\max}(t_i,g,\tau_i\setminus\{t_i\})\|\Tilde{e}_{t_i}\|_2^2\\
=&\sum_{t\in[d_1]}\sum_{\Tilde{\tau}_{t,i}\in \binom{[d_1]\setminus\{t\}}{d_0-1}}\sum_{g=g_0({t})}^{n_{t}} \frac{1}{\binom{d_1}{d_0}}Q_{\max}(t,g,\tilde{\tau}_{t,i})\|\tilde{e}_t\|_2^2
    =\sum_{t\in[d_1]}\beta_t \|\tilde{e}_t\|_2^2
    \end{aligned}
\end{equation*}
with $\beta_t = \sum\limits_{\Tilde{\tau}_{t,i}\in\binom{[d_1-1]}{d_0-1}}\sum\limits_{g=g_0({t})}^{n_{t_i}} \frac{1}{\binom{d_1}{d_0}}Q_{\max}(t,g,\tilde{\tau}_{t,i} )$.
\end{proof}
 Notice that
\begin{equation}\label{eqn:Qmax} \sum_{g=g_0(t)}^{n_t}\sum_{\ell=0}^{k}\frac{\binom{N_tp_{t,\ell}}{g}a^t_{g,\ell}}{\binom{N_t}{n_t}}\prod_{s\in\tau_i\setminus\{t\}}\frac{b^s_g}{\binom{N_s}{n_s}}\leq\sum_{g=g_0(t)}^{n_t}\mathbb{P}(t,g)\leq 1,
 \end{equation}
 we thus have $\sum_{g=g_{0}(t)}^{n_t}Q_{\max}(t,g,\tau_i\setminus\{t\}) \leq 1$, and
    \begin{align*}
\mathbb{E}_{\tau_{i}}\mathbb{E}_{t}\mathbb{E}_{\ell}\frac{e_{t,\ell}^2}{\|A_{t}\|_2^2}
    \leq&\sum_{\tau_i\in\binom{[d_1]}{d_0}}\frac{1}{\binom{d_1}{d_0}}\sum_{{t}\in\tau_i}\sum_{g=g_0({t})}^{n_{t_i}}Q_{\max}(t,g,\tau_i\setminus\{t\})\|\Tilde{e}_{t}\|_2^2
    \leq\frac{d_0}{d_1}\sum_{t=1}^{d_1}\|\tilde{e}_t\|_2^2.
    \end{align*}

\begin{proof}[The proof of Theorem~\ref{thm:multiple-main}]
    Combining Lemma~\ref{lem:ax-b} and Lemma~\ref{lem:err}, we thus have Theorem~\ref{thm:multiple-main}.
\end{proof}

Next we provide some remarks for our main result Theorem~\ref{thm:multiple-main}.
 
\begin{remark}
\begin{enumerate}
\item[(i)]  $0<1-\frac{d_0}{d_1}Q_{\min}\sigma_{\min}^2(\tilde{A})<1$, provided that  $d_0\leq d_2$. This results   from the facts that   $\sigma_{\min}^2(\tilde{A})\leq \frac{d_1}{d_2}$ (since $\sum_{i=1}^{d_2}\sigma_i^2(\tilde{A})=d_1$) and $0 \leq Q_{\min}\leq 1$ (by \eqref{eqn:Qmax}).
    \item[(ii)]  From \eqref{eq:mul_main_para}, one can see the relation between $\beta_t$ and $d_0$ is complicated.  
An example in Table \ref{tab:d0_effect} shows that increasing $d_0$, to some extent, can decrease $\beta_t$ and thus, improves the speed of convergence. For more details about   the optimal choice of  $d_0$, one can refer to Appendix~\ref{appendix:opt_d0}. Meanwhile, one should be aware of that  larger $d_0$ leads to more communication cost. Thus, in practice, finding an optimal $d_0$ is not just minimizing $\beta_t$ but also reducing the communication cost.
\item[(iii)]
When adversaries $\{\widetilde{e}_t\}_t$ are relatively small, the method without the block-list can guarantee a convergence error with the same or smaller magnitude as $\{\widetilde{e}_t\}_t$ using the right parameters. However, one can conclude from  \eqref{eq:mul_main_para}, when adversaries $\{\widetilde{e}_t\}_t$ are relatively larger, the convergence is not guaranteed. Therefore, it is crucial to introduce the block-list method for   excluding the adversarial workers.
 
\begin{table}[!h]
\caption{ Total number of workers $N_r \equiv 10$, number of error categories $k=3$, the adversarial rate $p_r = p/k$.}
    \label{tab:d0_effect}
    \centering
    \begin{tabular}{||c|c|c|c|c||}
         \hline
         $p$ &$n_r$ & $d_0 $ & $Q$ & $\beta_t$ \\
         \hline
         \multirow{3}{4em}{\quad $0.6$}
         & 5 &2 &$4.25\times 10^{-3}$ &$8.5\times10^{-4}$ \\
         & 5 &3 &$3.3\times 10^{-4}$ &$9.9\times 10^{-5}$ \\
         & 5 &5 &$3.63\times 10^{-6}$ &$1.82\times 10^{-6}$  \\
         \hline
    \end{tabular}
\end{table}
\end{enumerate}
\end{remark}
\subsection{Block-list method}
In this section, we use  $P_\ell$ to denote the proportion of workers in category $\ell$ among the total of $N$ workers. Thus, the number of workers in category $\ell$ is   $NP_\ell$. As per Algorithm~\ref{alg:multi_block_list}, after $S$ iterations, the worker $w^*$ with the highest count of non-mode selections, $c^+_{w^*}$, will be added to the block-list. To assess the effectiveness of the proposed method, it is crucial to calculate the probability that a bad or reliable worker is included in the block-list. This problem can be mathematically reformulated as follows.
\begin{problem}
Let $c^+_w(S)\coloneqq c^+_w$ (resp. $c^0_w(S)\coloneqq c^0_w$) be the counters of the worker $w$ being non-mode (resp.    mode or in no mode case) among $S$ iterations. Then we have $0\leq c^+_w,c^0_w \leq S$ and $\sum_{w=1}^{N }(c^+_w + c^0_w) = nS$. The probability that $w^*$ is in the block-list after S iterations can be calculated as follows:
\begin{equation}
    \begin{aligned}
    \mathbb{P}_{\textit{bl}}(w^*)
    = \mathbb{P}(c^+_{w^*} > c^+_w,\forall w\neq w^*) \text{ s.t. } \sum_{w=1}^{N }c^+_w + c^0_w = nS.
    \end{aligned}
\end{equation}
\end{problem}
Note that this probability can be calculated by using integer dynamic programming or estimated by  Monte Carlo simulations.\\
\begin{lemma}\label{lmm:pro_bl}
Run Alg.~\ref{alg:multi_block_list}  with $S$ iteration. 
 Then  the conditional probability that a reliable worker $w_0$ is  in the block-list is $
\frac{P_0\mathbb{P}_{\textit{bl}}(w_0)}{\sum_{i=0}^{k}P_{i}\mathbb{P}_{\textit{bl}}(w_{i})}$. 
Similarly, the conditional probability that a bad worker  $w_{\ell}$ in category $\ell$ is  in the block-list is $
\frac{P_{\ell}\mathbb{P}_{\textit{bl}}(w_{\ell})}{\sum_{i=0}^{k}P_{i}\mathbb{P}_{\textit{bl}}(w_{i})}$.
\end{lemma}
\begin{proof} First notice that we have  
 the following two facts:
(i)The probability that a reliable worker is in the block-list is
$P_0N\mathbb{P}_{\textit{bl}}(w_0)$. 
(ii) The probability that a bad worker $w_{\ell}$ in category $\ell$ is   in the block-list is
$P_{\ell}N\mathbb{P}_{\textit{bl}}(w_{\ell})$.  Therefore,
the probability  that a worker, either reliable or bad, i in the block-list is
$\sum_{i=0}^{k}P_{i}N\mathbb{P}_{\textit{bl}}(w_{i})$. The conditional probabilities can be easily computed by considering the ratios  $\frac{P_0N\mathbb{P}_{\textit{bl}}(w_0)}{\sum_{i=0}^{k}P_{i}N\mathbb{P}_{\textit{bl}}(w_{i})}$ and $\frac{P_{\ell}N\mathbb{P}_{\textit{bl}}(w_{\ell})}{\sum_{i=0}^{k}P_{i}N\mathbb{P}_{\textit{bl}}(w_{\ell})}$.
\end{proof}
 
To illustrate how the quantities changes with respect to $S$ in Lemma \ref{lmm:pro_bl}, we consider the following example.
\begin{example}
   Assume that there are two categories of workers i.e. $k = 0,1$, and 5 workers $w_i, i = 1,\ldots, 5$ in total with $w_1, w_2 \in C_1, w_3, w_4,w_5 \in C_0$. Let $n_r \equiv 3$. Note that $\mathbb{P}_{\textit{bl}}(w_1)=\mathbb{P}_{\textit{bl}}(w_2)\coloneqq \mathbb{P}_{\textit{bl}}^1$ and $\mathbb{P}_{\textit{bl}}(w_3)=\mathbb{P}_{\textit{bl}}(w_4)=\mathbb{P}_{\textit{bl}}(w_5)\coloneqq \mathbb{P}_{\textit{bl}}^0$. The probability is estimated by Monte Carlo simulations. We simulated the experiment 100 times and count the numbers of experiments where each worker is listed in the block-list. Those numbers are used to calculate the frequency and estimate the probability. The estimated results are summarized in  Table~\ref{tab:bl}. 
\begin{table}[!h]
    \caption{Conditional probability of being in block-list.}
    \label{tab:bl}
    \centering
    \begin{tabular}{||c|c|c|c|c||}
         \hline
         $S$ &$5$ & $10 $ & $50$ & $100$ \\
         \hline
         $\mathbb{P}_{\textit{bl}}^1$ &0.403 &$0.452$ &$0.5$ &$0.5$\\
         \hline
         $\mathbb{P}_{\textit{bl}}^0$ &$0.065$ &$0.032$ & $\sim 0$& $\sim 0$\\
         \hline
    \end{tabular}
\end{table}
  Table~\ref{tab:bl} shows that the probability of an adversarial worker in the block-list increases as the number of iterations $S$ increases. Meanwhile, the probability of a reliable worker in the block-list decreases. Using the method with the block-list, we are able to avoid choosing the results from the adversarial workers. As a results, the probability of using the adversarial workers decreases, i.e., $\beta_t$ decreases.
\end{example}
\section{Simulations}
In this section, we evaluate the performance of our approaches for solving consistent linear systems through simulations. We randomly generate a row-normalized matrix $A\in\mathbb{R}^{2400 \times 100}$ and a vector $x\in\mathbb{R}^{100}$, both from a normalized Gaussian distribution, and set $b = Ax$. For simplicity, we assume that each row has the same number of error categories $k$ and the same adversarial rate $p_r = p/k$, where $p$ is the total adversarial rate.
The linear system $Ax = b$ is solved using Algorithm~\ref{alg:multi_block_list} with and without the block-list. At each iteration, $d_0$ rows of $A$ are chosen uniformly at random, and for each row, $n_r$ workers are selected from the $N_r$ workers to participate in the calculation. 
We further assume  $N_r$ are the same and  $n_r$ are equal to $n$ for all $r$.
 The simulation results show how the number of used rows $d_0$, the number of used workers $n$, the total adversary rate $p$, and the number of error categories $k$ affect the performance.
\begin{figure}[!h]
    \centering
    \begin{subfigure}[t]{0.45\textwidth}
        \centering
         \includegraphics[width=\textwidth]{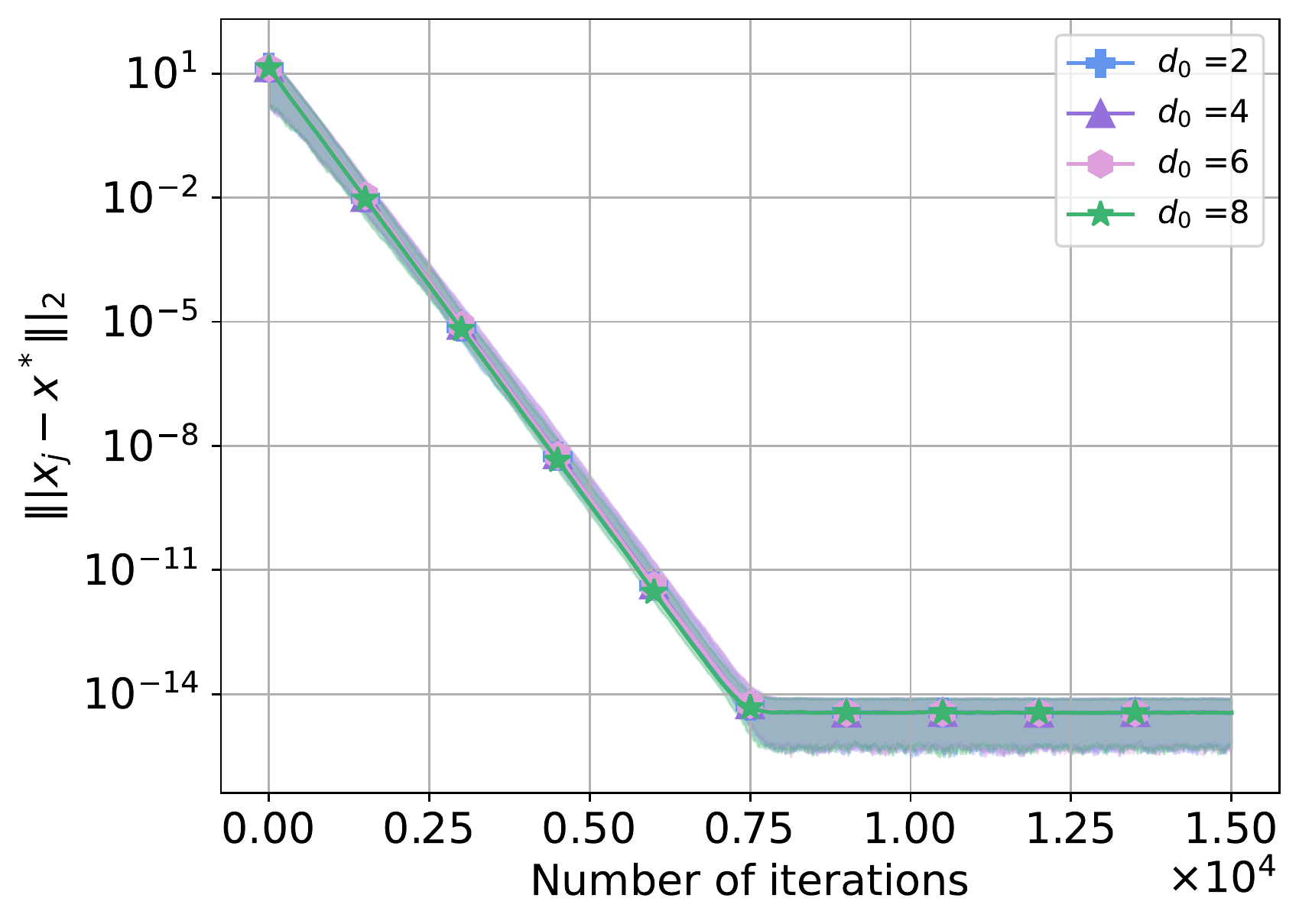}
         \caption{Error vs. $d_0$: $p=0.2$; using Alg.~\ref{alg:multi_block_list} with block-list.}
         \label{fig:varyds_0.2_wbl_small}
    \end{subfigure}
    \hspace{1em}
    \begin{subfigure}[t]{0.45\textwidth}
        \centering
         \includegraphics[width=\textwidth]{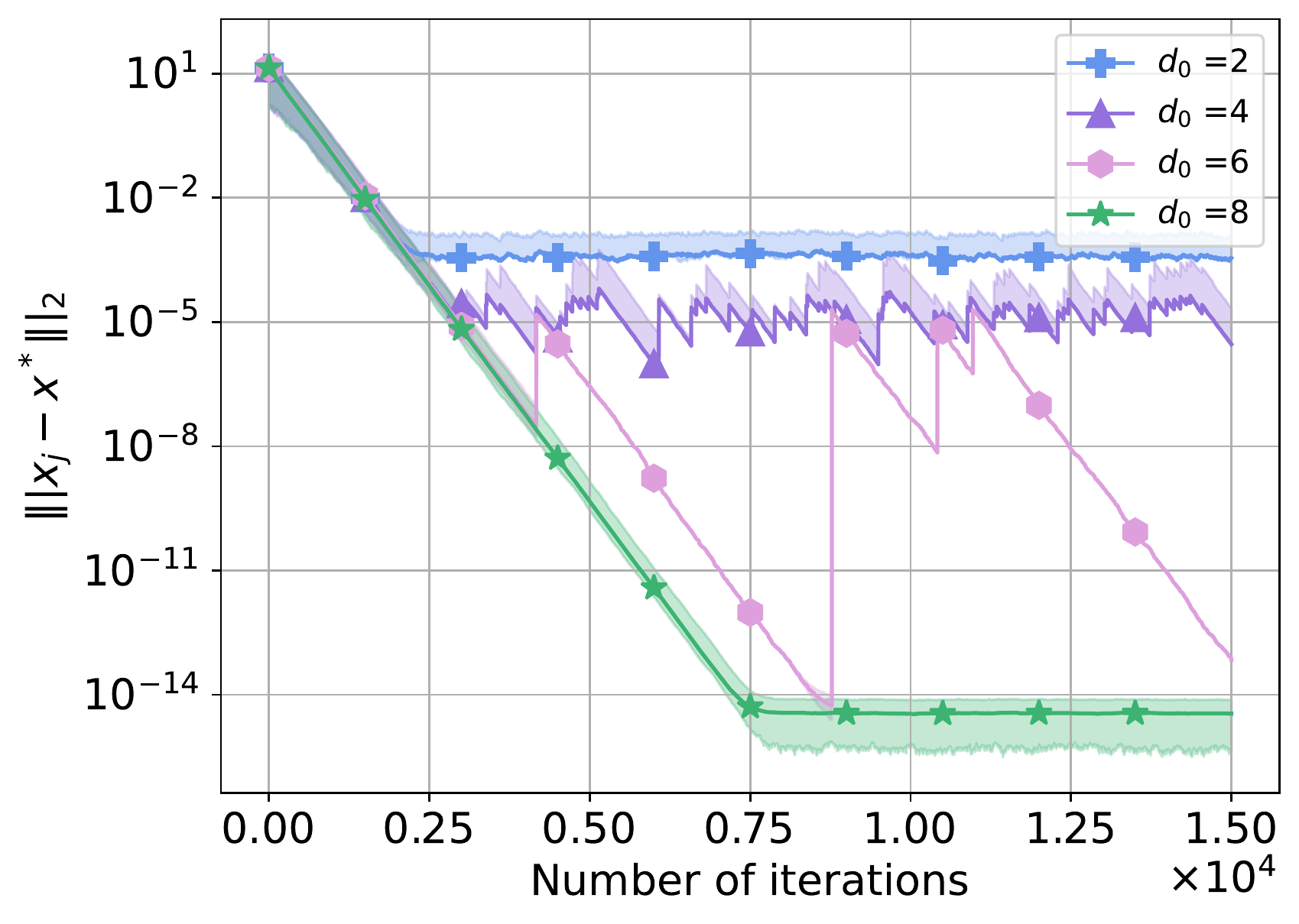}
         \caption{Error vs. $d_0$: $p=0.2$; using Alg.~\ref{alg:multi_block_list} without block-list.}
         \label{fig:varyds_0.2_wob_small}
    \end{subfigure}\\
    \centering
    \vspace{1em}
    \begin{subfigure}[t]{0.45\textwidth}
        \centering
         \includegraphics[width=\textwidth]{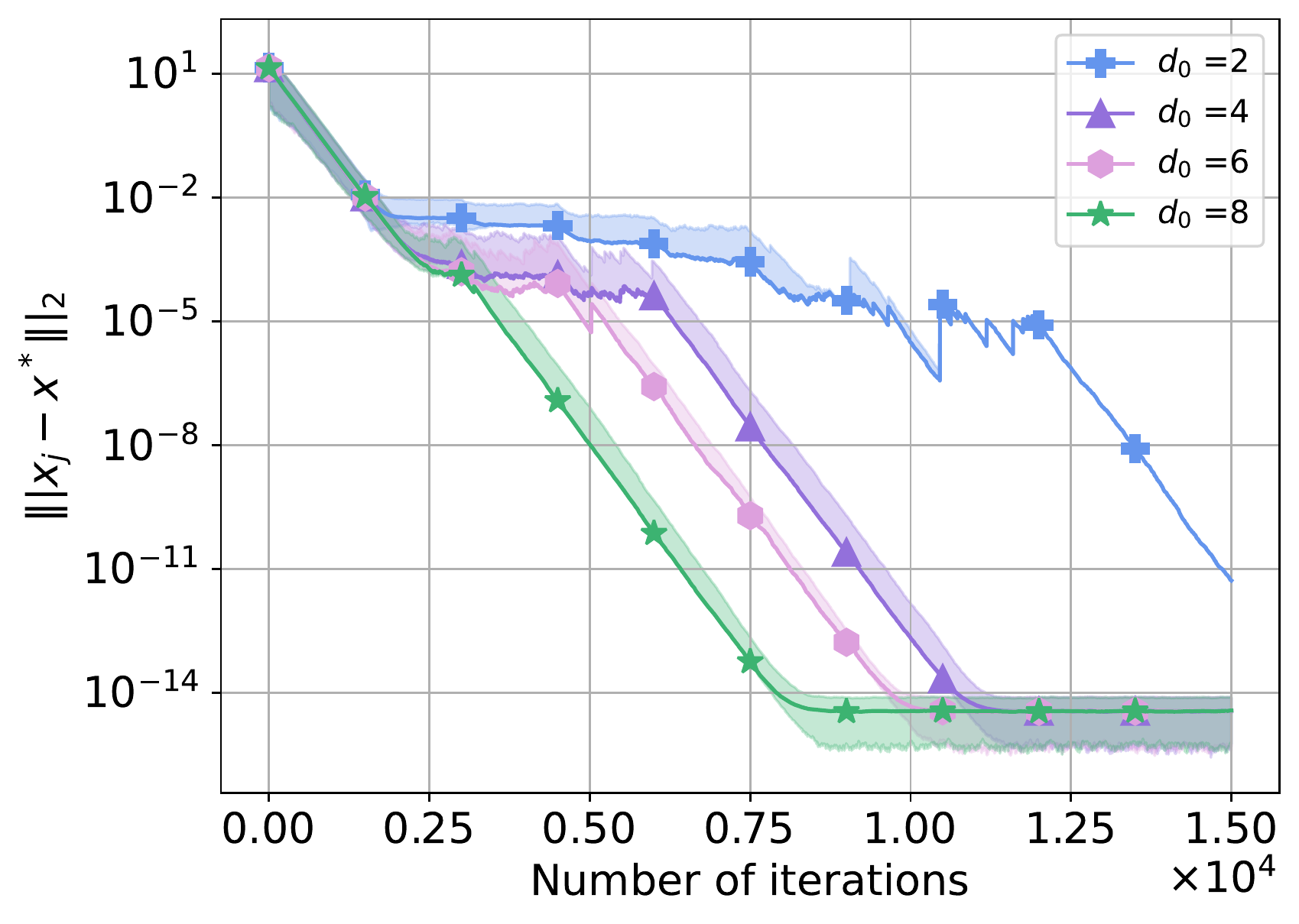}
         \caption{Error vs. $d_0$: $p=0.6$; using Alg.~\ref{alg:multi_block_list} with block-list.}
         \label{fig:varyds_0.6_wbl_small}
    \end{subfigure}
    \hspace{1em}
    \begin{subfigure}[t]{0.45\textwidth}
        \centering
         \includegraphics[width=\textwidth]{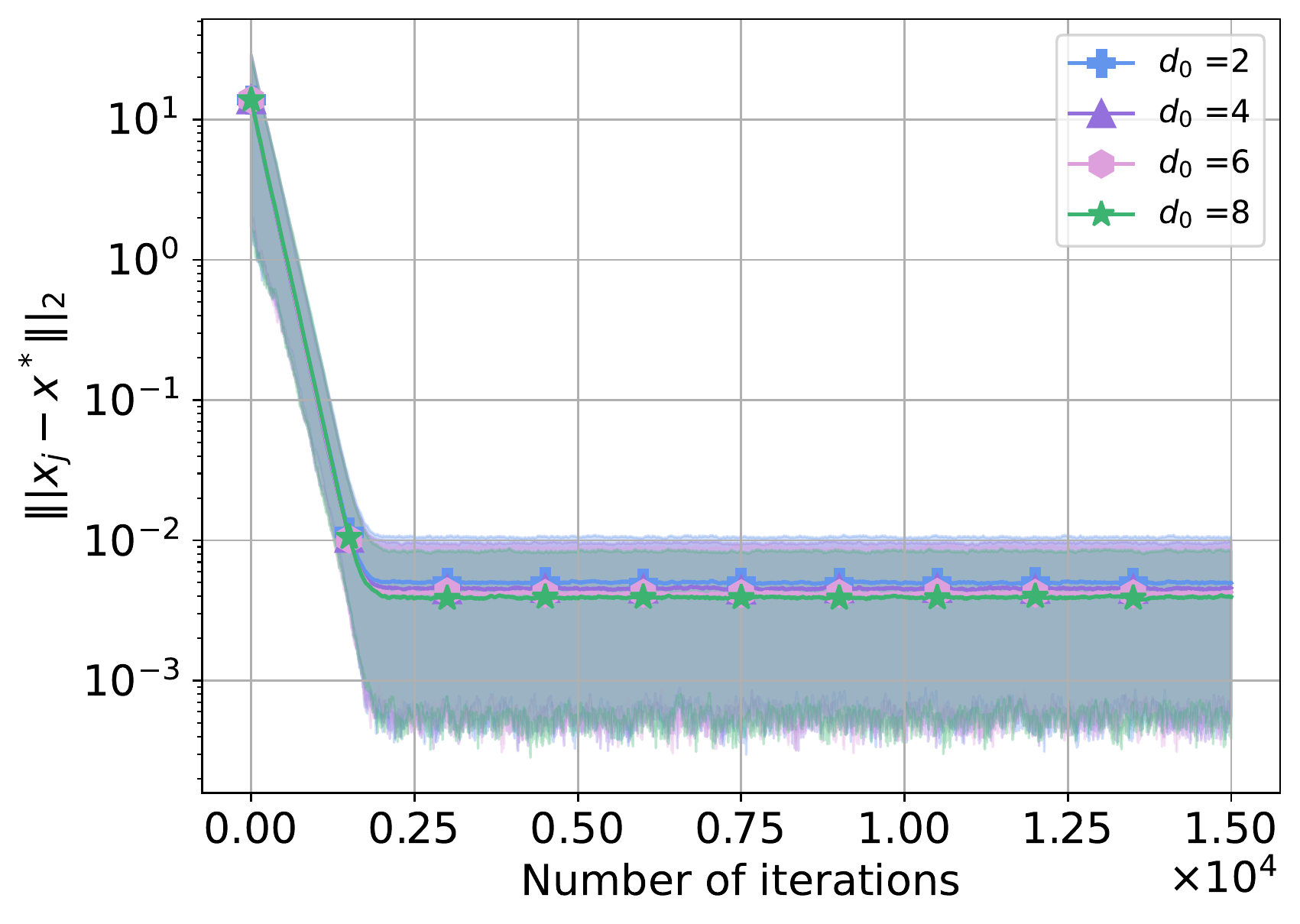}
         \caption{Error vs. $d_0$: $p=0.6$; using Alg.~\ref{alg:multi_block_list} without block-list.}
         \label{fig:varyds_0.6_wobl_small}
    \end{subfigure}
    \caption{Effects of the number of used rows $d_0$ on convergence: $N_r=20, n_r=4, k=3, \|e\|_{\infty} = 10^{-3}$. The error norms were averaged over $50$ trials (the solid lines) with $90\%$ percentiles (the shaded areas).}
    \label{fig:varyds_smallerr}
\end{figure}
\begin{figure}[!h]
    \centering
    \begin{subfigure}[t]{0.45\textwidth}
        \centering
         \includegraphics[width=\textwidth]{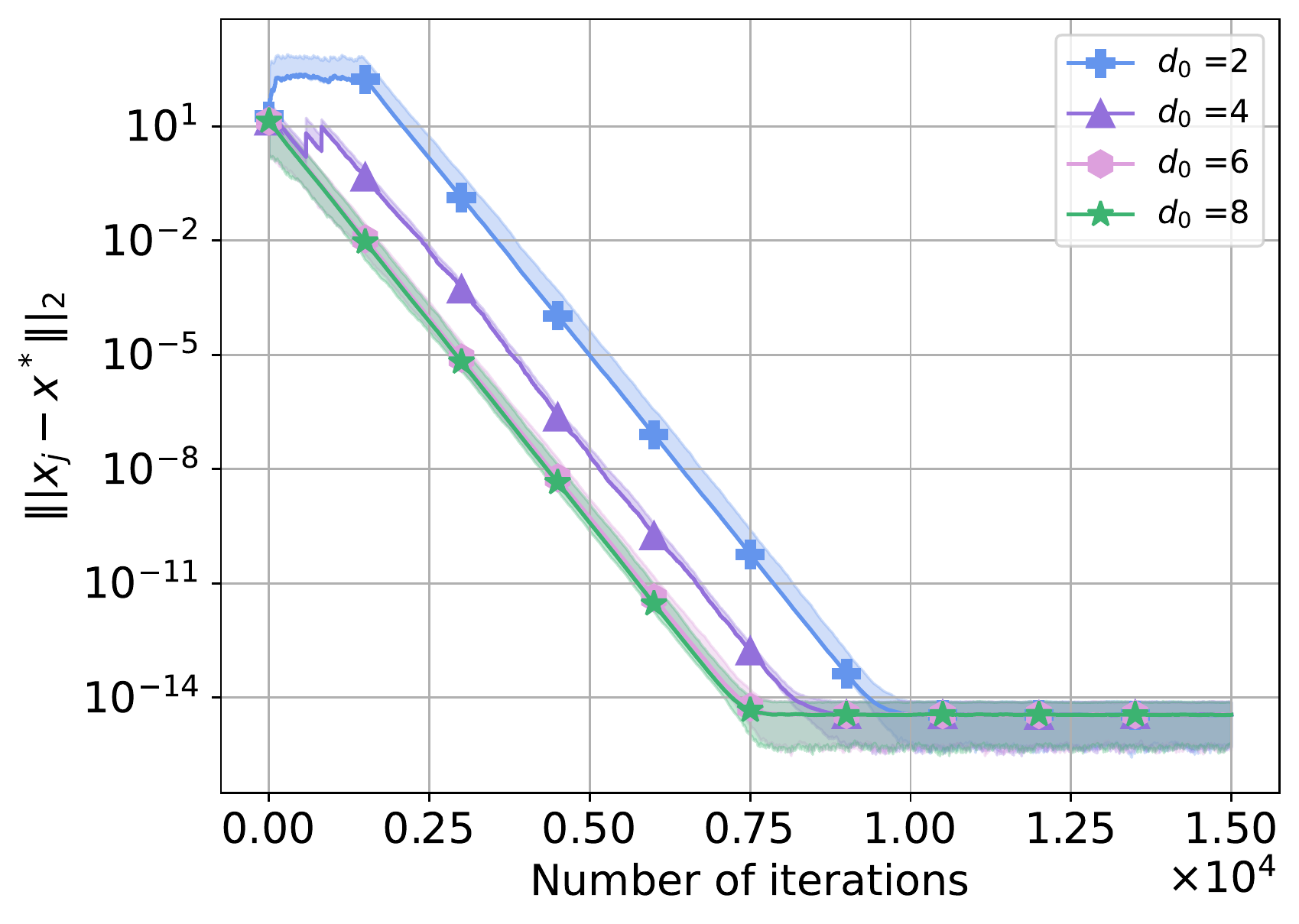}
         \caption{Error vs. $d_0$: $p=0.2$; using Alg.~\ref{alg:multi_block_list} with block-list.}
         \label{fig:varyds_0.2_wbl}
    \end{subfigure}
    \hspace{1em}
    \begin{subfigure}[t]{0.45\textwidth}
        \centering
         \includegraphics[width=\textwidth]{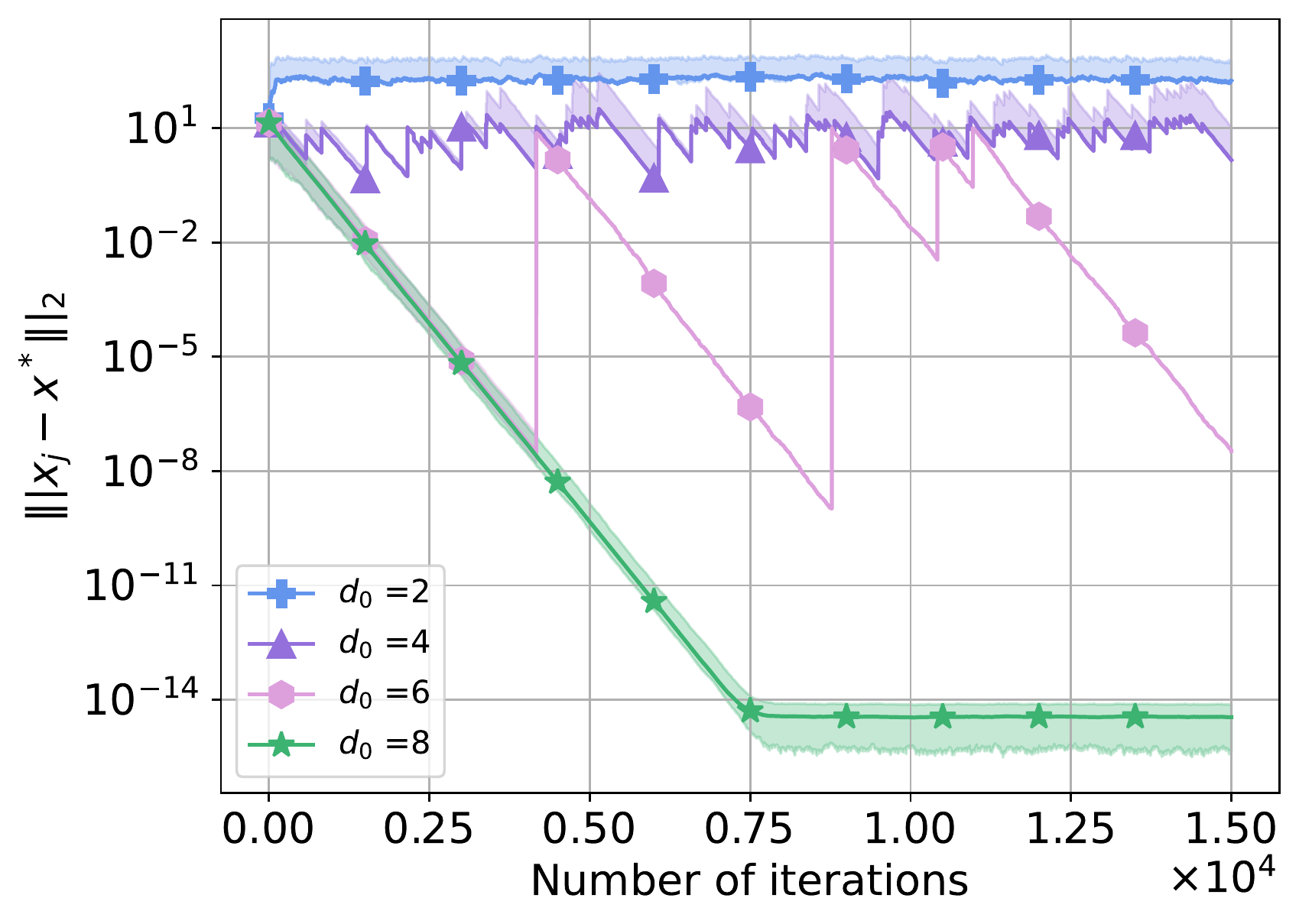}
         \caption{Error vs. $d_0$: $p=0.2$; using Alg.~\ref{alg:multi_block_list} without block-list.}
         \label{fig:varyds_0.2_wobl}
    \end{subfigure}\\
    \centering
    \vspace{1em}
    \begin{subfigure}[t]{0.45\textwidth}
        \centering
         \includegraphics[width=\textwidth]{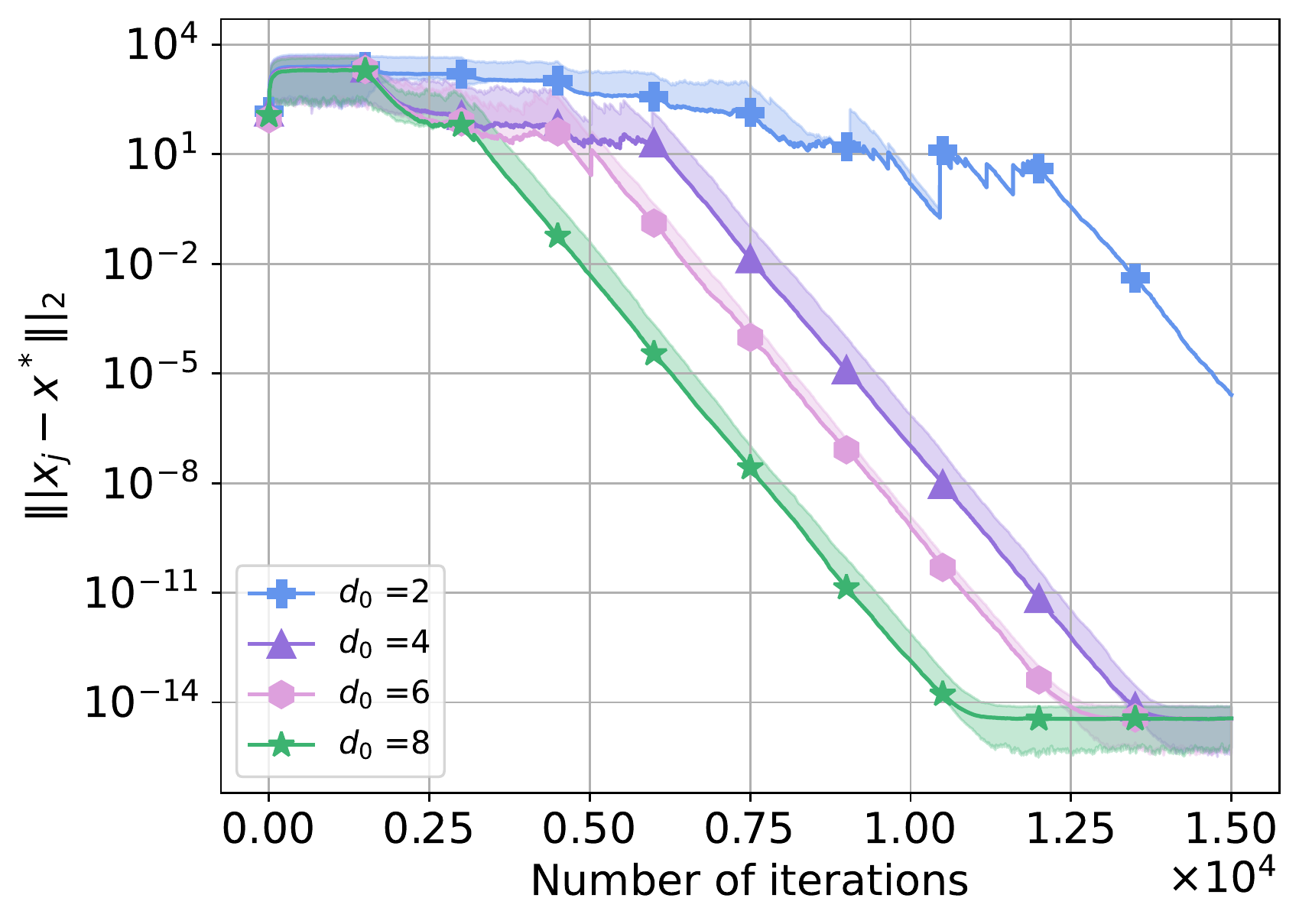}
         \caption{Error vs. $d_0$: $p=0.6$; using Alg.~\ref{alg:multi_block_list} with block-list.}
         \label{fig:varyds_0.6_wbl}
    \end{subfigure}
    \hspace{1em}
    \begin{subfigure}[t]{0.45\textwidth}
        \centering
         \includegraphics[width=\textwidth]{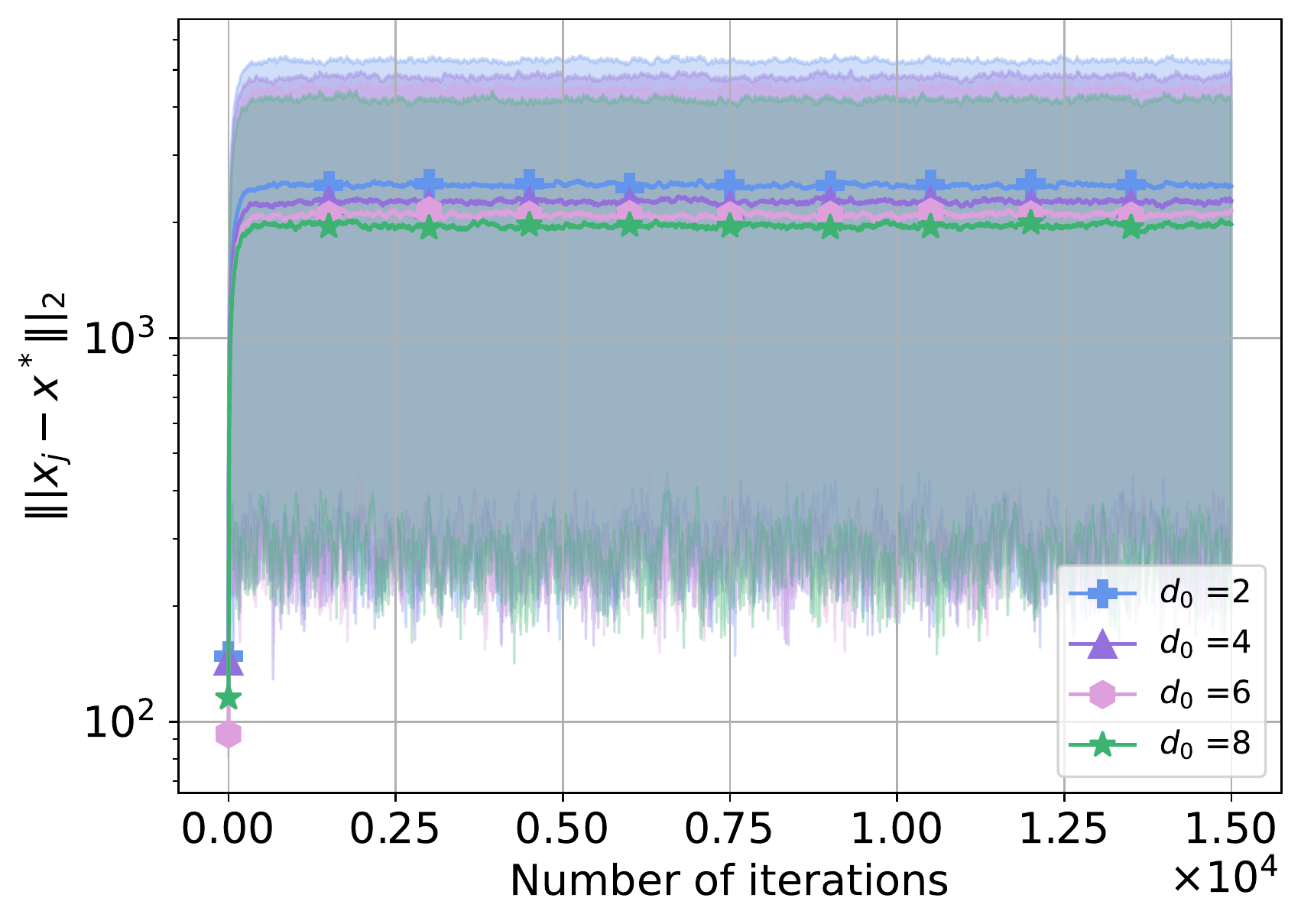}
         \caption{Error vs. $d_0$: $p=0.6$; using Alg.~\ref{alg:multi_block_list} without block-list.}
         \label{fig:varyds_0.6_wobl}
    \end{subfigure}
    \caption{Effects of different data sizes $d_0$ on convergence: $N_r=20, n_r=4, k=3$, and $\|e\|_{\infty} = 500$. The error norms were averaged over $50$ trials (the solid lines) with $90\%$ percentiles (the shaded areas).}
    \label{fig:varyds_bigerr}
\end{figure}
\begin{figure}[!h]
     \centering
      \begin{subfigure}[t]{0.45\textwidth}
        \centering
         \includegraphics[width=\textwidth]{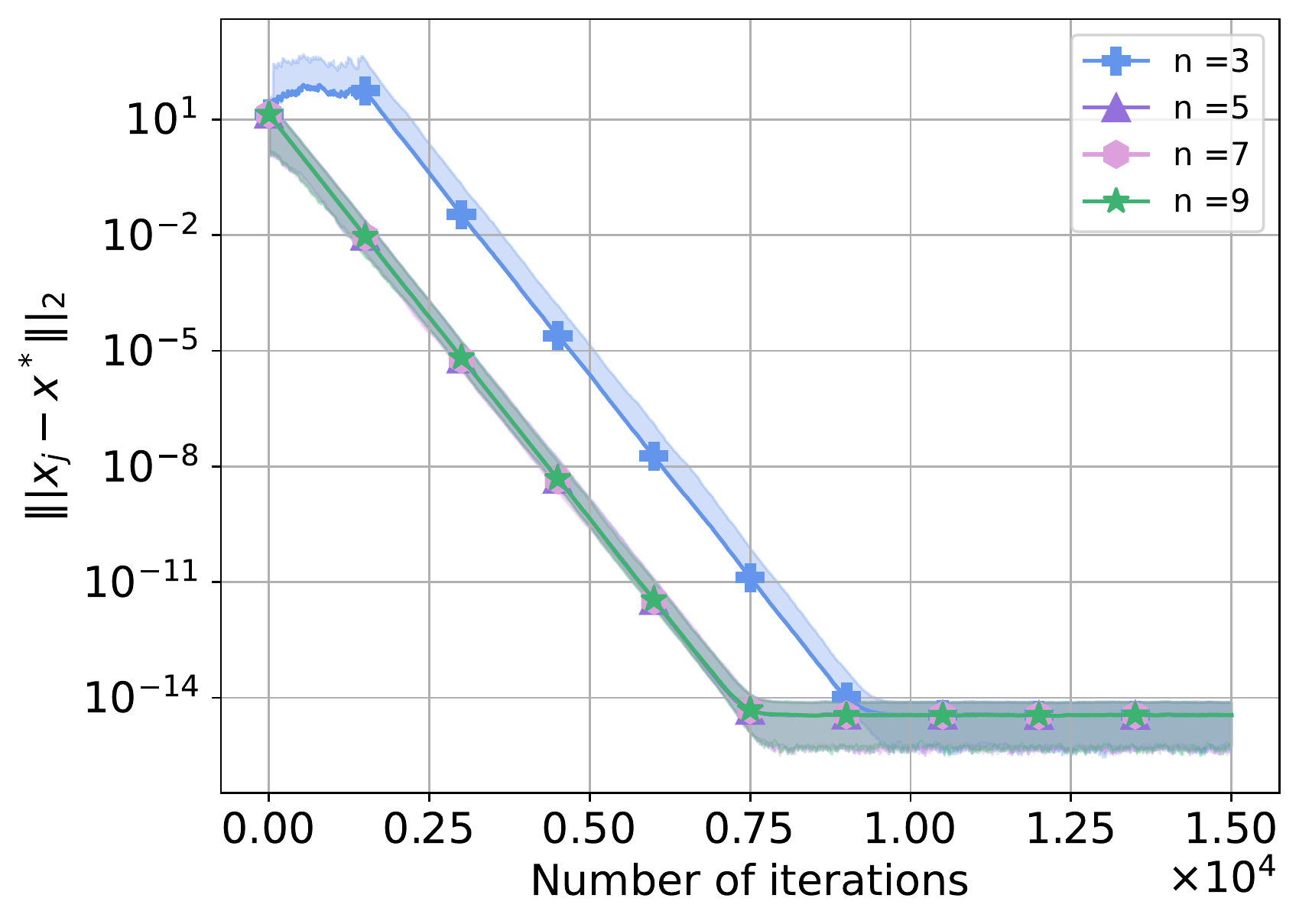}
         \caption{Error vs. $n$: $p=0.2$; using Alg.~\ref{alg:multi_block_list} with block-list.}
         \label{fig:varyn_0.2_w}
    \end{subfigure}
    \hspace{1em}
    \begin{subfigure}[t]{0.45\textwidth}
        \centering
         \includegraphics[width=\textwidth]{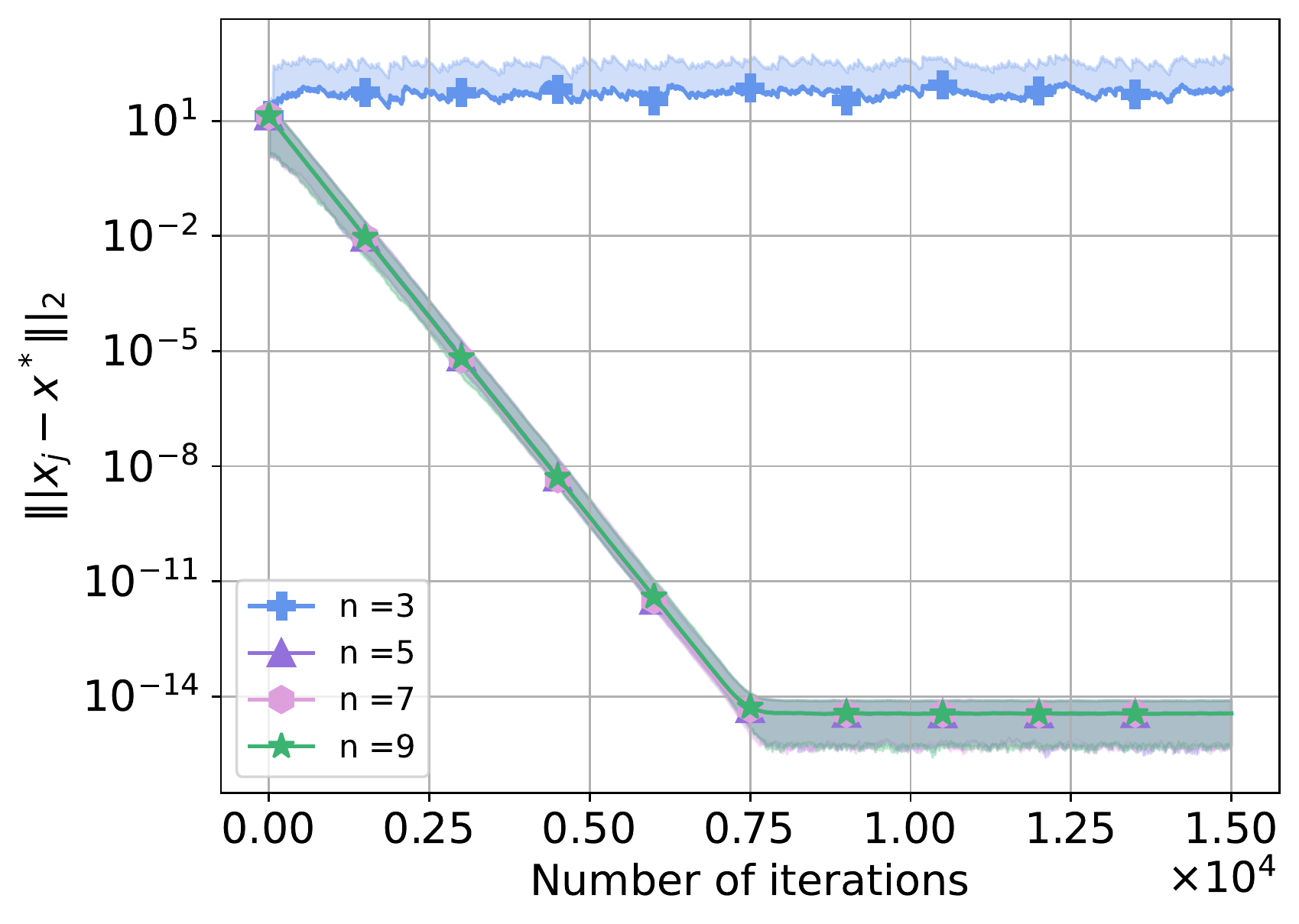}
         \caption{Error vs. $n$: $p=0.2$; using Alg.~\ref{alg:multi_block_list} without block-list.}
         \label{fig:varyn_0.2_wo}
    \end{subfigure}\\
      \hspace{1em}
      \begin{subfigure}[t]{0.45\textwidth}
        \centering
         \includegraphics[width=\textwidth]{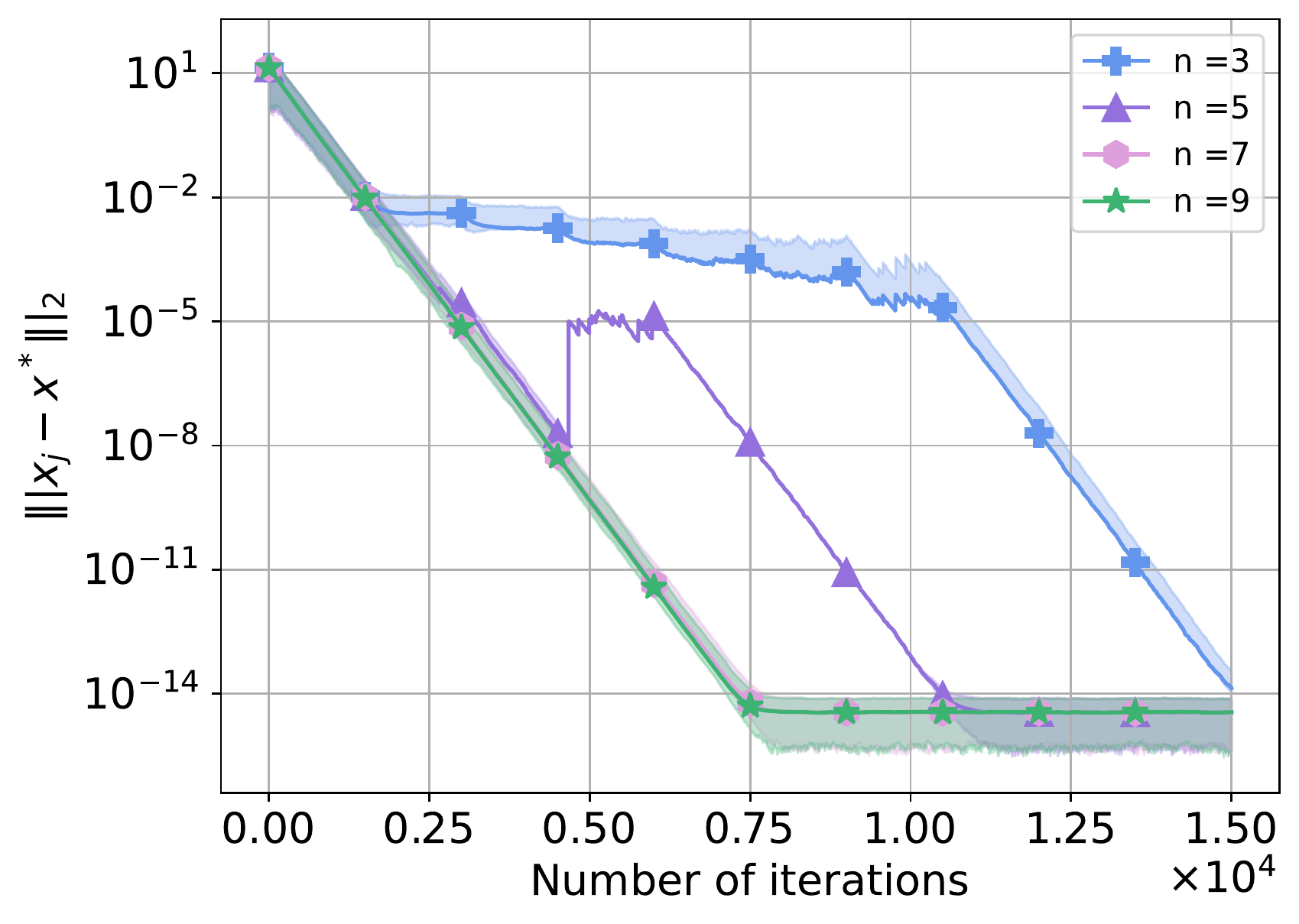}
         \caption{Error vs. $n$: $p=0.6$; using Alg.~\ref{alg:multi_block_list} with block-list.}
         \label{fig:varyn_0.6_w}
    \end{subfigure}
    \hspace{1em}
    \begin{subfigure}[t]{0.45\textwidth}
        \centering
         \includegraphics[width=\textwidth]{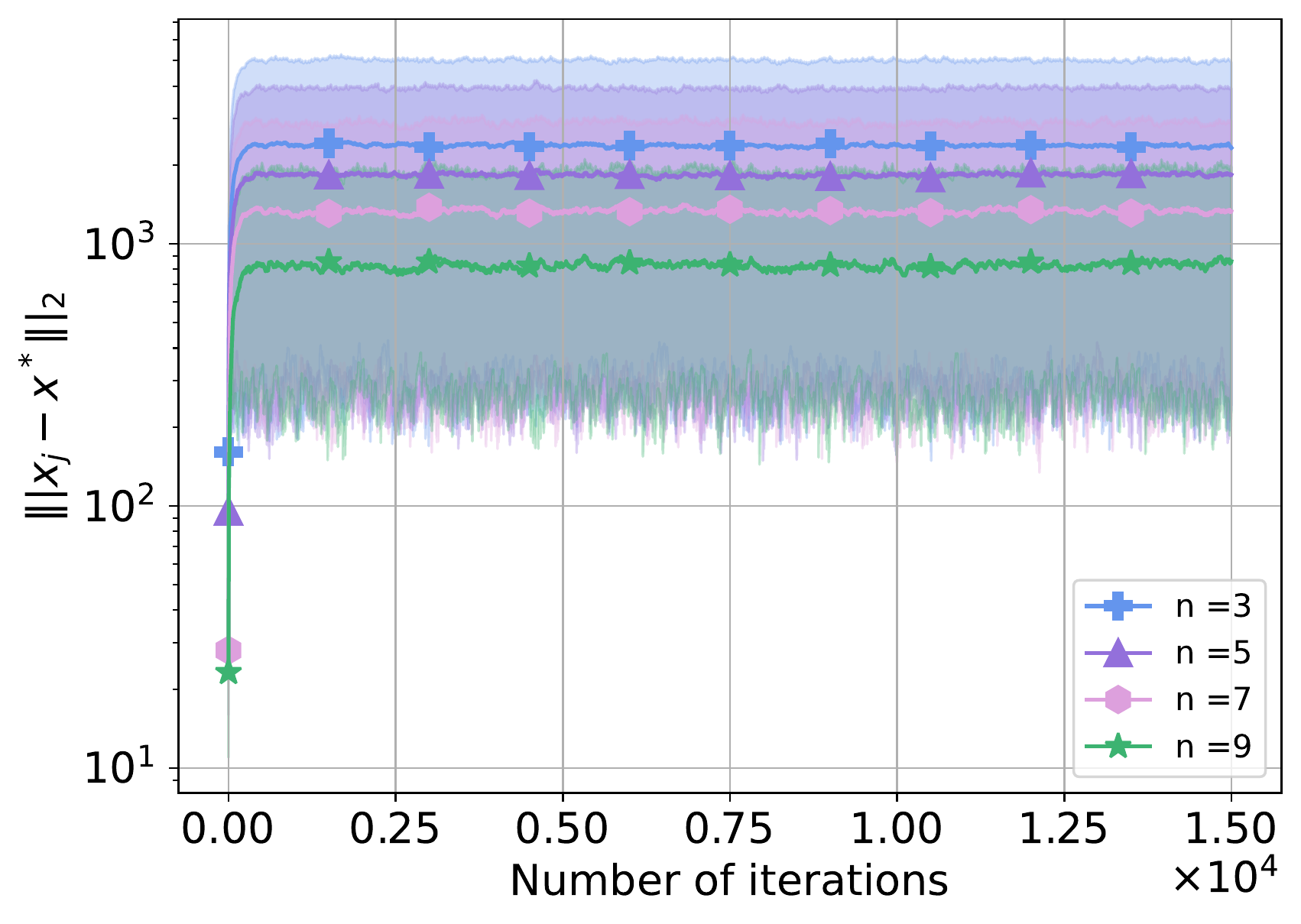}
         \caption{Error vs. $n$: $p=0.6$; using Alg.~\ref{alg:multi_block_list} without block-list.}
         \label{fig:varyn_0.6_wo}
    \end{subfigure}
    \caption{Effects of the number of used workers $n_r$ for row $r$ (note that $n_r \equiv n$ for all selected rows): $k=3$, $d_0=6$, $N_r=20$, $\|e\|_{\infty}=500$.}
    \label{fig:varyn}
     \vspace{-3mm}
\end{figure}

\begin{figure}
     \centering
      \begin{subfigure}[t]{0.45\textwidth}
        \centering
         \includegraphics[width=\textwidth]{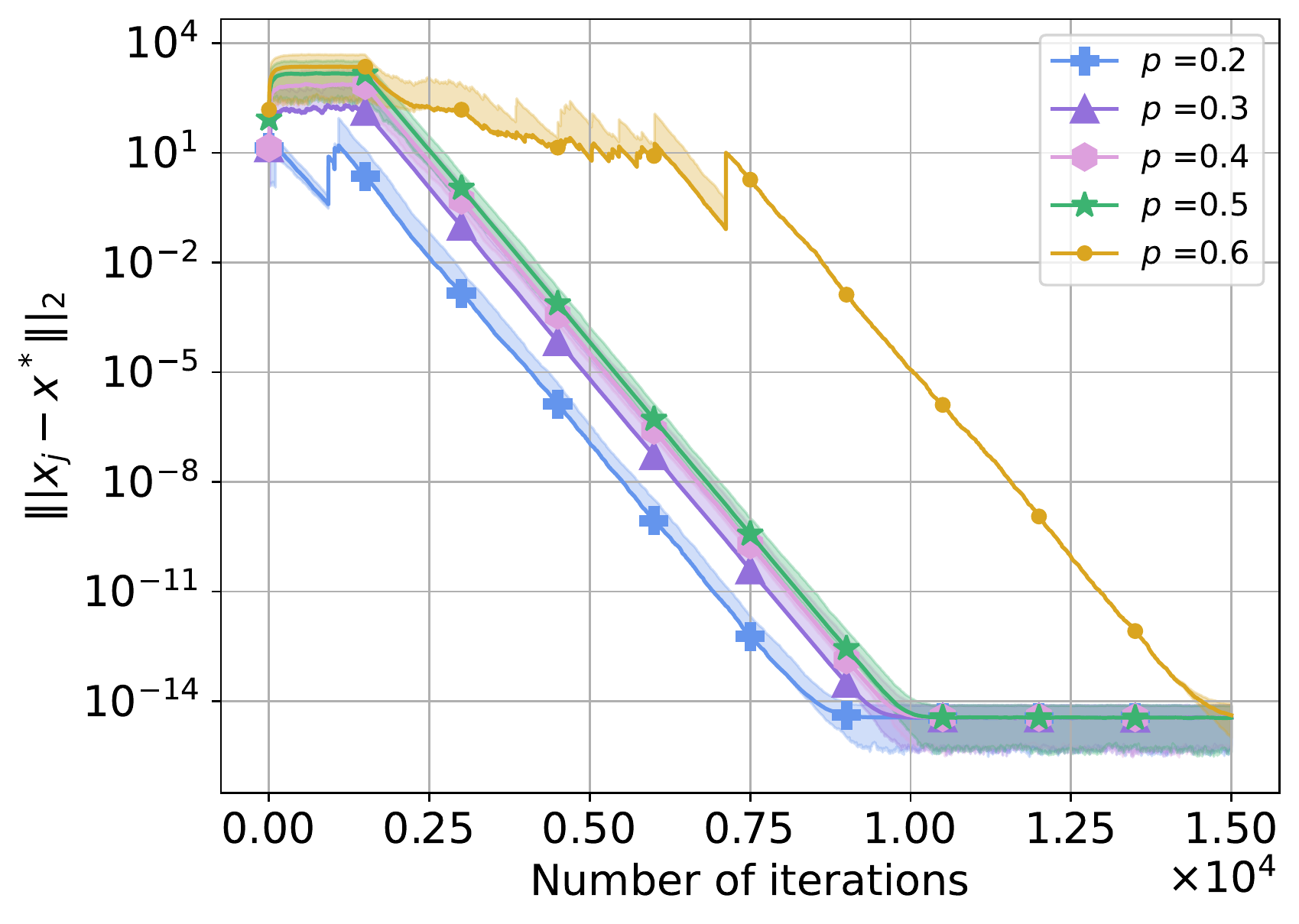}
         \caption{Error vs. $p$ :with block-list}
         \label{fig:varyp_w}
    \end{subfigure}
    \hspace{1em}
    \begin{subfigure}[t]{0.45\textwidth}
        \centering
         \includegraphics[width=\textwidth]{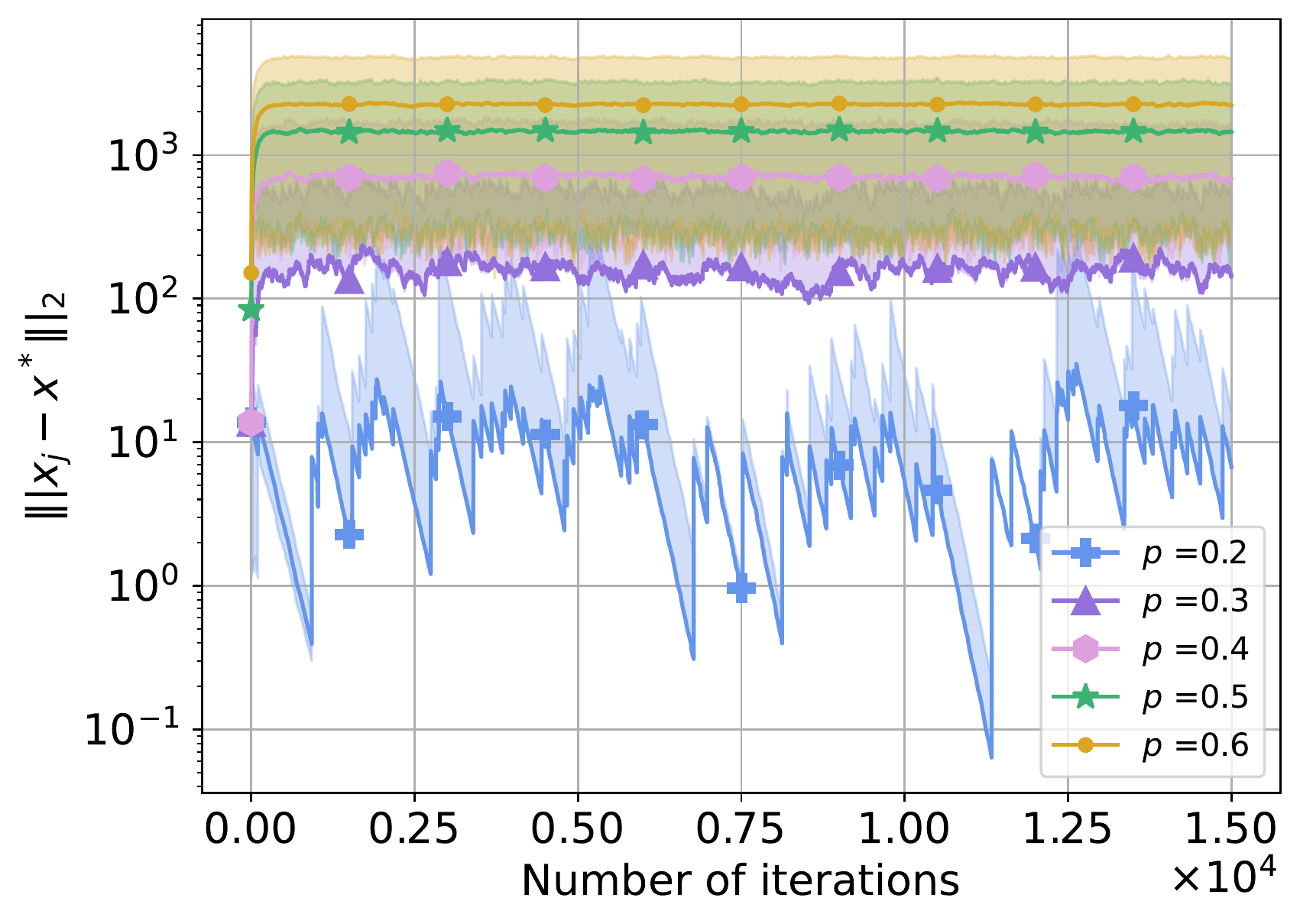}
         \caption{Error vs. $p$: without block-list}
         \label{fig:varyp_wo}
    \end{subfigure}
\caption{Effects of the adversarial rate $p$, $d_0=3, N_r=20, n_r=4$, $\|e\|_{\infty} = 5\times10^{2}$, and $k=3$. Squared error norms were averaged over $50$ trials with the $90\%$ percentiles}
    \label{fig:varyp}
     \vspace{-3mm}
\end{figure}

\begin{figure}
     \centering
      \begin{subfigure}[t]{0.45\textwidth}
        \centering
         \includegraphics[width=\textwidth]{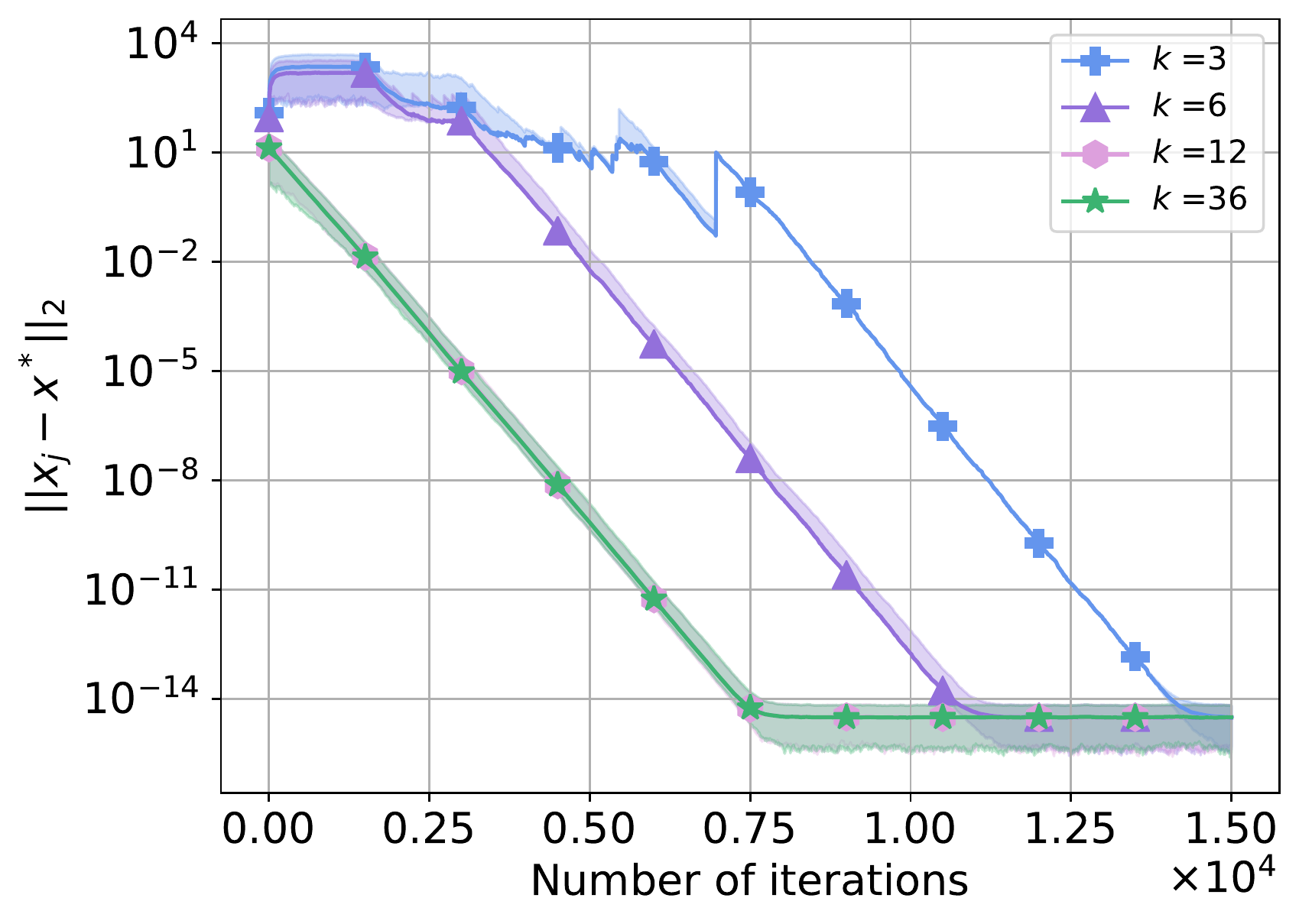}
         \caption{Error vs. $k$: $p=0.6$; using Alg.~\ref{alg:multi_block_list} with block-list.}
    \end{subfigure}
    \hspace{1em}
    \begin{subfigure}[t]{0.45\textwidth}
        \centering
         \includegraphics[width=\textwidth]{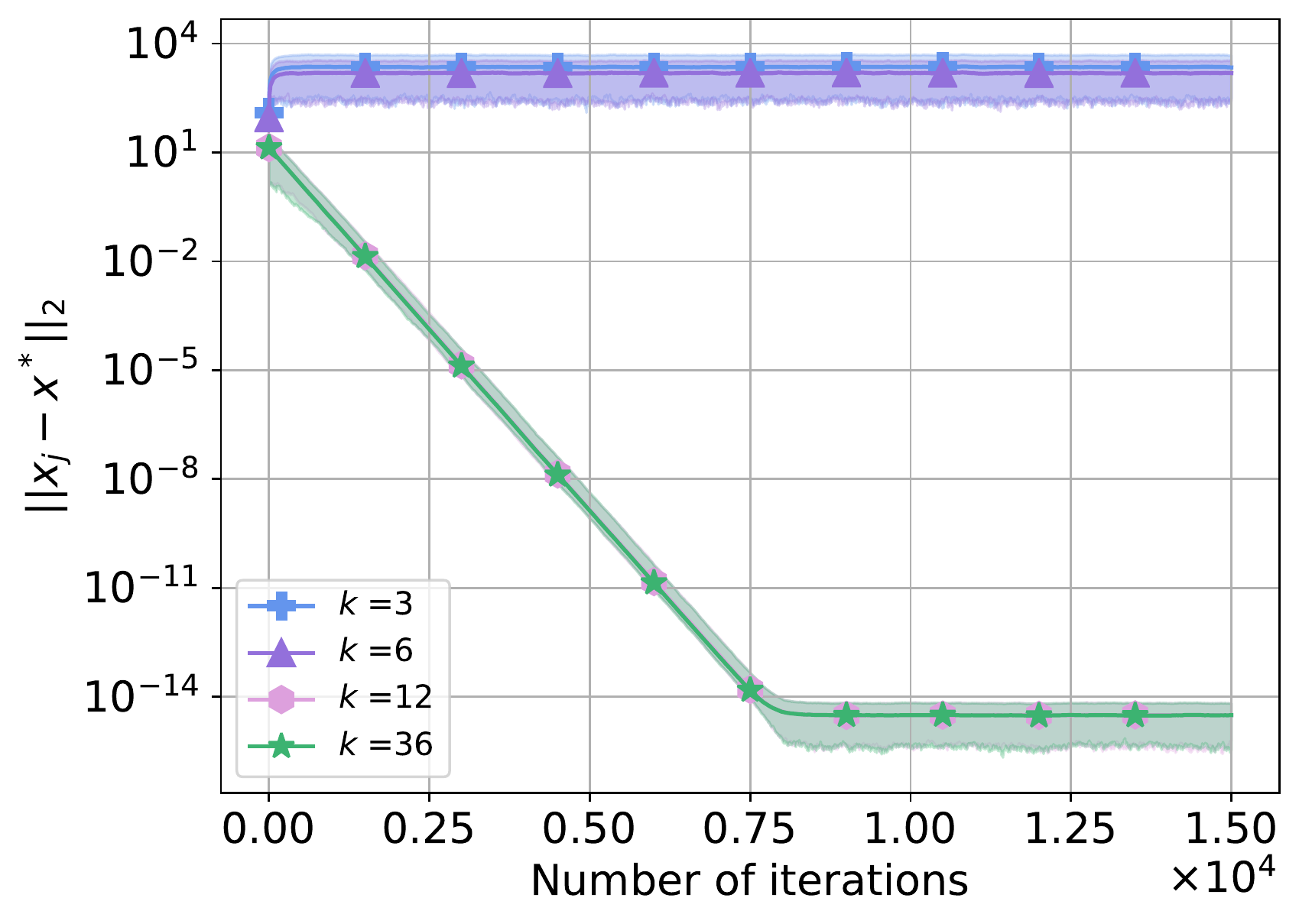}
         \caption{Error vs. $k$: $p=0.6$; using Alg.~\ref{alg:multi_block_list} without block-list.}
    \end{subfigure}
    \caption{Effects of the number of categories $k$. $N_r=20,n_r=4, d_0 = 4, \|e\|_{\infty}=5\times 10^2$}
    \label{fig:varyk}
     \vspace{-3mm}
\end{figure}

Figs.~\ref{fig:varyds_smallerr} and \ref{fig:varyds_bigerr} illustrate the impact of the number of used rows $d_0$ on the convergence results of our distributed RK method with and without the block-list. In this example, increasing the number of used rows $d_0$ from $2$ to $4$ improves the convergence rate for both with and without block-list scenarios, regardless of the magnitude of the adversaries $e$.
Fig.~\ref{fig:varyds_smallerr} demonstrates the convergence results when   $\|e\|_{\infty} = 10^{-3}$. From the figures, Alg.~\ref{alg:multi_block_list}  with the block-list shows fast convergence over all choices of $d_0$ when the adversarial rate $p=0.2$ (Fig.~\ref{fig:varyds_0.2_wbl_small}); the larger the number of used rows $d_0$, the faster the convergence when the adversarial rate $p=0.6$ (Fig.~\ref{fig:varyds_0.6_wbl_small}).
Fig.~\ref{fig:varyds_0.2_wob_small} reveals that, when $d_0=2,4$, the central worker may use a corrupted step-size, resulting in oscillations around the solution and the error converging to a range of magnitude between $10^{-3}$ and $10^{-5}$. On the other hand, when $d_0=6,8$, the error goes to 0 after 7500 iterations. 
As seen in Fig.~\ref{fig:varyds_0.6_wobl_small}, the convergence errors of all $d_0$ values lie in the range of $(10^{-4} , 10^{-2})$. With a small $\|e\|_{\infty}$, the method without the block-list can converge by increasing $d_0$. However, when the magnitude of the adversaries and the adversarial rate are large (as shown in Fig.~\ref{fig:varyds_0.6_wobl}), the convergence is not guaranteed without a block-list. In comparison, the method with the block-list converges quickly to an accuracy of $10^{-14}$ when  $d_0$ is sufficiently large (as seen in Fig.~\ref{fig:varyds_0.2_wbl}). This highlights the importance of using the block-list in an environment with larger outliers. 

Figure~\ref{fig:varyn} examines the impact of the number of chosen workers $n_r$ on the convergence, with adversary rates of $0.2$ and $0.6$. As the number of workers increases from $3$ to $7$, the convergence becomes faster in both with and without block-list   methods. However, the method using the block-list generally provides better convergence compared to the method without  block-list. The block-list method requires extra storage, but it is worth the trade-off in terms of improved convergence. Without the block-list, oscillations can be observed in the convergence when $p=0.2$ and $n_r \equiv n=3$ and when $p=0.6$ and all $n_r$.

Fig.~\ref{fig:varyp} demonstrates the effect of the adversary rate on the convergence. As the adversary rate $p$ increases, the accuracy decreases. Even though the adversary rate is large,  the final results using the block-list method are still satisfying. Without the block-list, when the adversarial rate $p >0.5$, the central worker fails to approach the true solution due to the adversarial workers. This again shows the importance and effectiveness of using the block-list, especially in a highly hostile environment with a higher adversarial rate and a higher magnitude of the adversary. In addition, Fig.~\ref{fig:varyk} shows the effect of the number $k$ of error categories    with the block-list. One can see that  our method converges when $k$ is big enough for the adversarial rate being  $0.2$ and $0.6$.
In Fig.~\ref{fig:wbc_varyds}, we use the Wisconsin (Diagnostic) Breast Cancer data set, which includes data points whose features are computed from a digitized image of a fine needle aspirate (FNA) of a breast mass and describe characteristics of the cell nuclei present in the image (see \cite{wcb_uci} for more details). Similar to the setup in \cite{haddock2020quantile},  we set the simulations in the following way: the collection of data points forms matrix $A \in \mathbb{R}^{569\times 10}$. We then normalize $A$ and construct $x$ and $b$ using a Gaussian distribution to form a consistent system. The convergence results in Fig.~\ref{fig:wbc_varyds} show the effectiveness of our method solving this linear systems in a relatively safer environment with an adversarial rate $p=0.3$ (Fig.~\ref{fig:wbc_varyds_0.3}) and a more hostile environment with an adversarial rate $p=0.6$ (Fig.~\ref{fig:wbc_varyds_0.6}). When $p=0.3$, the method converges within $1000$ iterations, and as $d_0$ increases, the convergence speed becomes faster. Meanwhile, when $p=0.3$, the method converges within $1500$ iterations, and $d_0=8$ has the fastest convergence speed among all choices of $d_0$.

Finally, we have investigated the impact of updating cycles $S$ on the accuracy of block-list recognition. 
 In Table~\ref{tab:diffS}, we calculate the accuracy of the block-list method when $S = 200, 500, 1000, 2000$ using the same matrix $A$ as in Fig. \ref{fig:varyds_smallerr}. The two examples in the table show that as $S$ increases, the accuracy is higher. 
  
\begin{figure}[!h]
     \centering
      \begin{subfigure}[t]{0.45\textwidth}
        \centering
         \includegraphics[width=\textwidth]{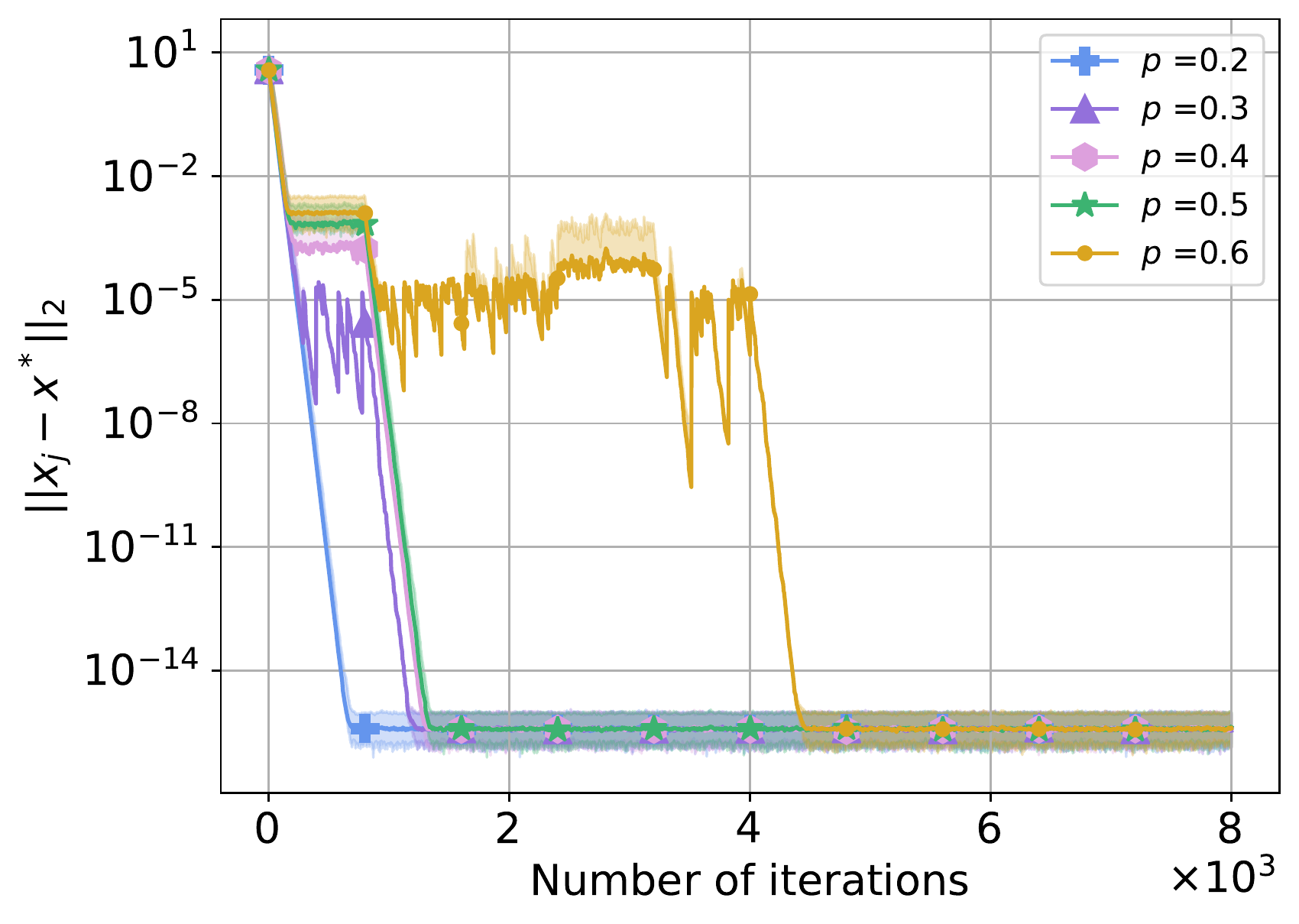}
         \caption{Error vs. $p$, with block-list.}
         \label{fig:wbc_varyds_0.3}
    \end{subfigure}
    \hspace{1em}
    \begin{subfigure}[t]{0.45\textwidth}
        \centering
         \includegraphics[width=\textwidth]{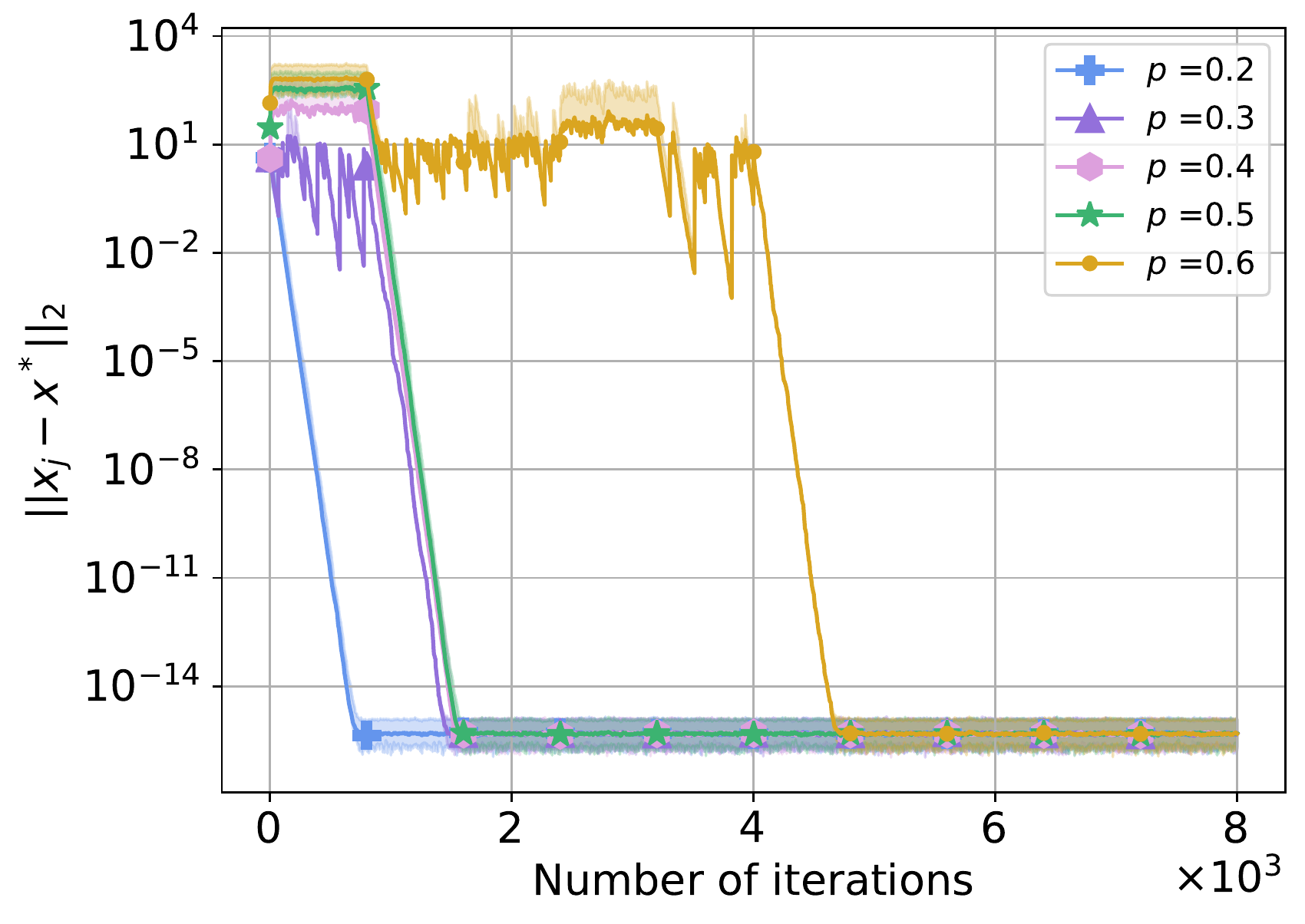}
         \caption{Error vs. $p$, without block-list.}
         \label{fig:wbc_varyds_0.6}
    \end{subfigure}
\caption{Effects of adversarial rate $p$ using the Breast Cancer Wisconsin data set, $N_r = 10, n_r=4, k=3$, $\|e\|_{\infty} = 500$.}
    \label{fig:wbc_varyds}
     \vspace{-3mm}
\end{figure}

\begin{table}[h]
    \centering
    \caption{The accuracy of recognizing the block-list with different updating cycles $S$ by fixing $k=3,N_r=20$, and $n_r=4$.}
    \begin{tabular}{||c|c|c|c|c||}
    \hline
        $S$ & $200$ & $500$ &$1000$ &$2000$ \\
        \hline
         $p=0.6,d_0=8$ &$0.75$ &$0.792$ &$0.875$ &$0.875$\\
         \hline
         $p=0.4,d_0=6$ &$0.75$ &$0.9375$ &$1$ &$1$\\
         \hline
    \end{tabular}
    \label{tab:diffS}
\end{table}
\section{Conclusion and future work}
It is of great significance for optimization algorithms to be robust and resistant to adversaries. 
In this work, we propose efficient algorithms based on the \textbf{mode} for solving large-scale linear systems in the presence of the adversarial workers. This kind of adversary has plenty of applications in the real world, e.g.  Internet of Things (IoT). We provide theoretical convergence guarantee and the theories are supported by our experiments. 
The methods are capable of handling various levels of adversarial rates. In particular, the method with the block-list is able to provide accurate estimation of solution when the adversarial rate $p > 0.5$, and at the same time, our method can identify the adversarial workers. Our experiments also highlight the impact of several important parameters of the adversaries and of anti-adversary strategies, namely, the involved row number  $d_0$ to update the solution per iteration, the number of error categories $k$, the adversary rate $p$, and the number of chosen workers $n_r$ at each iteration. 

Our method can also be adapted to solve the nonlinear problem. In Fig.~\ref{fig:al1}, we applied the method with the block-list to solve the optimization problem with $\ell_1$ regularization $\frac{1}{2}\|Ax-b\|_2^2 + \gamma\|x\|_{1}$ with $\gamma = 1$.  In the distributed setting, the problem can be formulated  as 
$\sum_{i=1}^{d_1}f_i(x)$, where $f_i(x) = \frac{1}{2}(A_ix_j - b_i)^2+\frac{\gamma}{d_1}\|x\|_1$.  Fig.~\ref{fig:al1_varyds} shows the distance between the solution using our method and the solution using the LASSO solver from the Python package scikit-learn \cite{scikit-learn} when choosing different number of rows $d_0$.  When the number of used rows $d_0 = 4$, it takes the least iterations to converge. All choices of $d_0$ end up with an relative error around $10^{-5}$ with the block-list. The object function $F(x)$ v.s. the iterations are shown in Fig.~\ref{fig:al1_obj}. 

\begin{figure}[t]
    \centering
    \begin{subfigure}{0.45\textwidth}
        \centering
        \includegraphics[width=\textwidth]{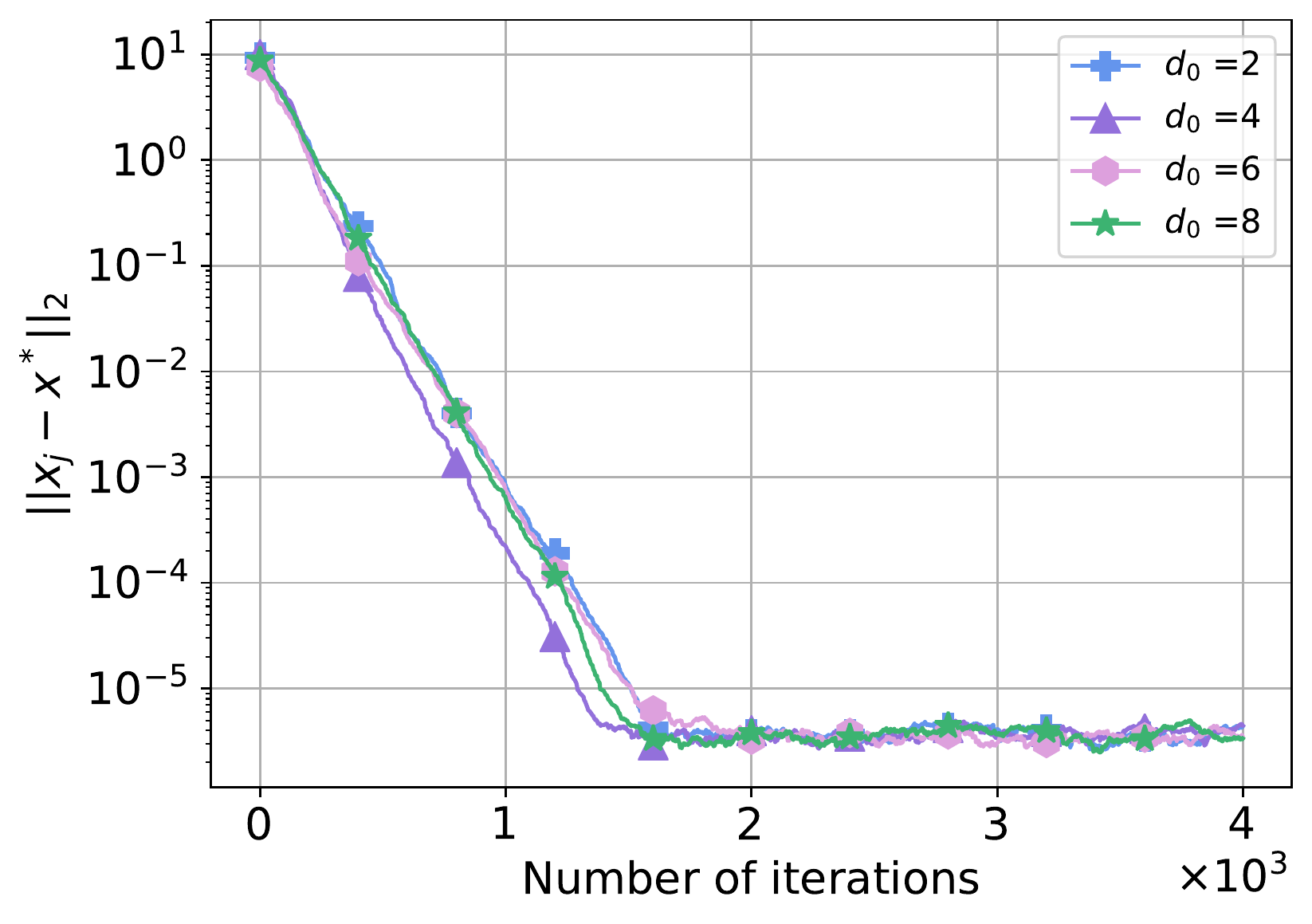}
     \caption{Error vs. number of used rows $d_0$: without block-list.}
     \label{fig:al1_varyds}
     \end{subfigure}
     \begin{subfigure}{0.45\textwidth}
        \centering
        \includegraphics[width=\textwidth]{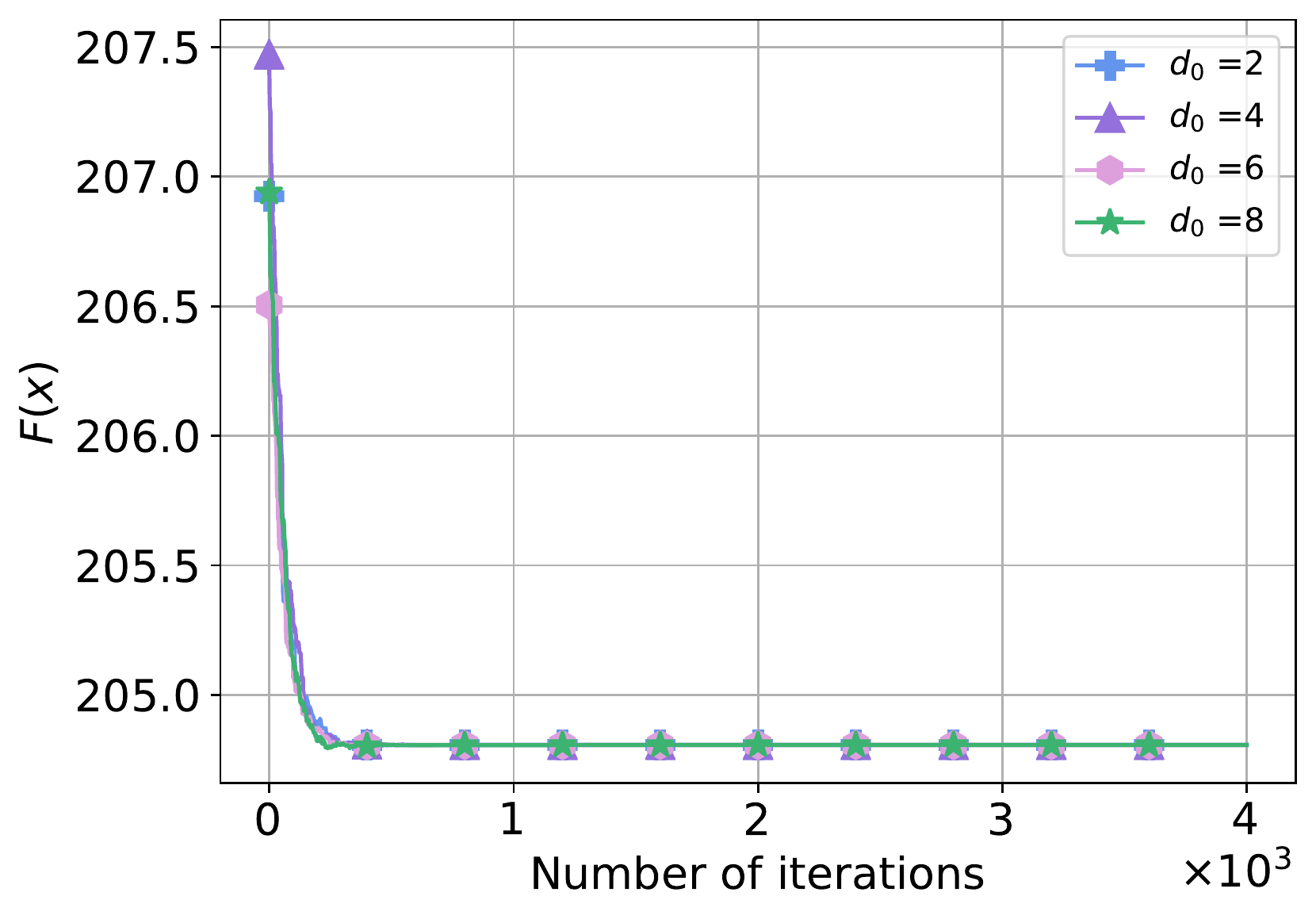}
     \caption{$F(x)$ vs. number of used rows $d_0$: without block-list.}
     \label{fig:al1_obj}
     \end{subfigure}
     \caption{Effects of the number of used rows $d_0$ on convergence: $ p=0.4$, $ \|e\|_{\infty} = 1$, $N_r = 20$, $n_r = 15$, and $\gamma=1$. The object function $F(x) = \sum_{i=1}^{d_1} \left(\frac{1}{2}(A_ix_j - b_i)^2 + \frac{\gamma}{d_1} \|x\|_1 \right)$.}
    \label{fig:al1}
\end{figure}

 \begin{figure}[t]
    \centering
        \includegraphics[width=0.45\textwidth]{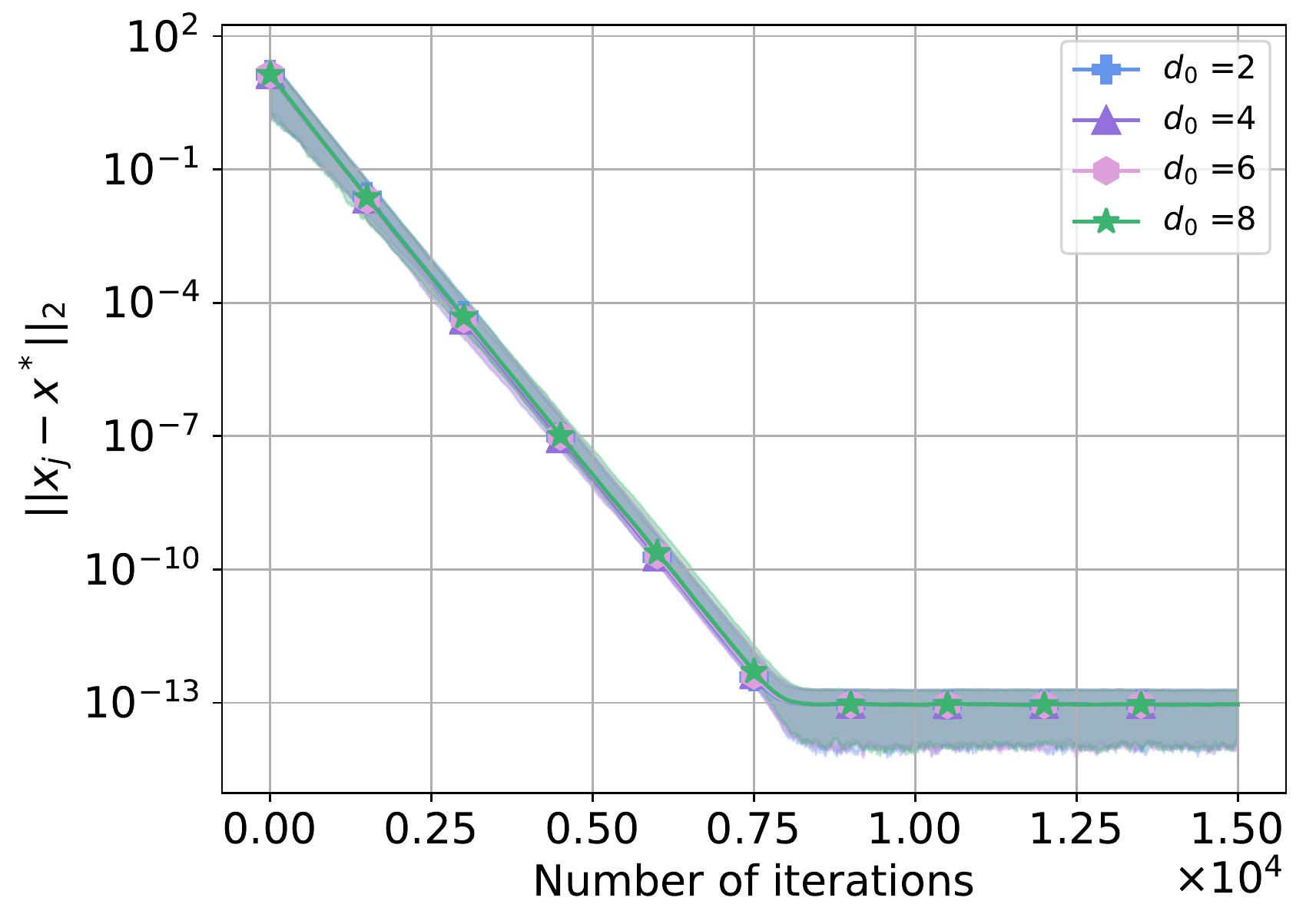}
     \caption{Effects of the number of used rows $d_0$ on convergence: $ p=0.6$, $ \|e\|_{\infty} = 10^{-3}$, $N_r = 20$, $n_r = 4$, $r = 3$, and $\alpha=0.7$ using RrDR method.}
    \label{fig:rrdr}
\end{figure}
Additionally, we provide convergence analysis for methods when multiple rows are selected uniformly at each iteration.  We also provide the proof of a more general sampling scheme where the row is sampled according to its squared Euclidean norm, although the proof only applies to the case where only a single row is used for computation at each iteration. 

Interesting avenues for future work include proving convergence under different sampling rules and rigorously generalizing this method to non-linear optimization problems. Another direction is to compare other iterative methods in the distributed and adversarial setting. For example, randomized r-sets-Douglas-Rachford (RrDr) \cite{han2022randomized} is a widely used method for solving linear systems of equations. This method usually shows a fast convergence compared to RK. The performance of the method is also tested in the adversarial distributed setting and the results are reported in Fig.~\ref{fig:rrdr}. Both methods start to converge after about 8000 iterations. However, for all $d_0$, using RrDR converges as fast as when $d_0 = 8$ using RK, see Fig.~\ref{fig:varyds_0.6_wbl_small}.
 
\section*{Acknowledgements} Authors are listed in alphabetical order. Some of the work for this article was done while  Longxiu Huang was Assistant Adjunct Professor and Xia Li  was a graduate student at UCLA.

\bibliographystyle{siamplain}
\bibliography{refer}

\begin{thebibliography}{10}

\bibitem{agmon_1954_mm}
{\sc S.~Agmon}, {\em The relaxation method for linear inequalities}, Canadian Journal of Mathematics, 6 (1954), p.~382–392, \url{https://doi.org/10.4153/CJM-1954-037-2}.

\bibitem{alistarh2018byzantine}
{\sc D.~Alistarh, Z.~Allen-Zhu, and J.~Li}, {\em Byzantine stochastic gradient descent}, in Advances in Neural Information Processing Systems, S.~Bengio, H.~Wallach, H.~Larochelle, K.~Grauman, N.~Cesa-Bianchi, and R.~Garnett, eds., vol.~31, Curran Associates, Inc., 2018.

\bibitem{bai2019partially}
{\sc Z.-Z. Bai and W.-T. Wu}, {\em On partially randomized extended {K}aczmarz method for solving large sparse overdetermined inconsistent linear systems}, Linear Algebra and Its Applications, 578 (2019), pp.~225--250.

\bibitem{bai2021}
{\sc Z.-Z. Bai and W.-T. Wu}, {\em On greedy randomized augmented {K}aczmarz method for solving large sparse inconsistent linear systems}, SIAM Journal on Scientific Computing, 43 (2021), pp.~A3892--A3911, \url{https://doi.org/10.1137/20M1352235}.

\bibitem{bitar2020stochastic}
{\sc R.~Bitar, M.~Wootters, and S.~E. Rouayheb}, {\em Stochastic gradient coding for flexible straggler mitigation in distributed learning},  (2019), pp.~1--5, \url{https://doi.org/10.1109/ITW44776.2019.8989328}.

\bibitem{borowski2015understanding}
{\sc H.~P. Borowski and J.~R. Marden}, {\em Understanding the influence of adversaries in distributed systems}, in 2015 54th IEEE Conference on Decision and Control (CDC), IEEE, 2015, pp.~2301--2306.

\bibitem{Chi2019MedianTruncatedGD}
{\sc Y.~Chi, Y.~Li, H.~Zhang, and Y.~Liang}, {\em Median-truncated gradient descent: A robust and scalable nonconvex approach for signal estimation}, 2019.

\bibitem{wcb_uci}
{\sc D.~Dua and C.~Graff}, {\em {UCI} machine learning repository}, 2017.

\bibitem{gao2020backdoor}
{\sc Y.~Gao, B.~G. Doan, Z.~Zhang, S.~Ma, J.~Zhang, A.~Fu, S.~Nepal, and H.~Kim}, {\em Backdoor attacks and countermeasures on deep learning: A comprehensive review}, arXiv preprint arXiv:2007.10760,  (2020).

\bibitem{data_poison}
{\sc J.~Geiping, L.~Fowl, W.~R. Huang, W.~Czaja, G.~Taylor, M.~Moeller, and T.~Goldstein}, {\em Witches' brew: Industrial scale data poisoning via gradient matching}, Clinical Orthopaedics and Related Research, abs/2009.02276 (2020).

\bibitem{evation}
{\sc I.~Goodfellow, P.~McDaniel, and N.~Papernot}, {\em Making machine learning robust against adversarial inputs}, Communications of the ACM, 61 (2018), p.~56–66, \url{https://doi.org/10.1145/3134599}.

\bibitem{GORDON1970471}
{\sc R.~Gordon, R.~Bender, and G.~T. Herman}, {\em Algebraic reconstruction techniques ({ART}) for three-dimensional electron microscopy and {X}-ray photography}, Journal of Theoretical Biology, 29 (1970), pp.~471--481.

\bibitem{gordon1975image}
{\sc R.~Gordon, G.~T. Herman, and S.~A. Johnson}, {\em Image reconstruction from projections}, Scientific American, 233 (1975), pp.~56--71.

\bibitem{haddock2020quantile}
{\sc J.~Haddock, D.~Needell, E.~Rebrova, and W.~Swartworth}, {\em Quantile-based iterative methods for corrupted systems of linear equations}, SIAM Journal on Matrix Analysis and Applications, 43 (2022), pp.~605--637, \url{https://doi.org/10.1137/21M1429187}.

\bibitem{han2022randomized}
{\sc D.~Han, Y.~Su, and J.~Xie}, {\em Randomized douglas-rachford method for linear systems: Improved accuracy and efficiency}, arXiv preprint arXiv:2207.04291,  (2022).

\bibitem{herman1993algebraic}
{\sc G.~T. Herman and L.~B. Meyer}, {\em Algebraic reconstruction techniques can be made computationally efficient (positron emission tomography application)}, IEEE transactions on medical imaging, 12 (1993), pp.~600--609.

\bibitem{karczmarz1937angenaherte}
{\sc S.~Kaczmarz}, {\em Angen$\ddot{a}$herte {A}ufl$\ddot{o}$sung von {S}ystemen linearer {G}leichungen}, Bulletin International de l'Académie Polonaise des Sciences et des Lettres,  (1937), pp.~355--357.

\bibitem{karakus2017straggler}
{\sc C.~Karakus, Y.~Sun, S.~Diggavi, and W.~Yin}, {\em Straggler mitigation in distributed optimization through data encoding}, in Advances in Neural Information Processing Systems, NIPS'17, 2017, p.~5440–5448.

\bibitem{sesame}
{\sc K.~Krishna, G.~S. Tomar, A.~P. Parikh, N.~Papernot, and M.~Iyyer}, {\em Thieves on sesame street! model extraction of bert-based apis}, Clinical Orthopaedics and Related Research, abs/1910.12366 (2019), \url{https://arxiv.org/abs/1910.12366}.

\bibitem{byz_General}
{\sc L.~Lamport, R.~Shostak, and M.~Pease}, {\em The {B}yzantine generals problem}, ACM Transactions on Programming Languages and Systems, 4 (1982), p.~382–401, \url{https://doi.org/10.1145/357172.357176}.

\bibitem{lamport2019byzantine}
{\sc L.~Lamport, R.~Shostak, and M.~Pease}, {\em The byzantine generals problem}, in Concurrency: the works of leslie lamport, 2019, pp.~203--226.

\bibitem{leventhal2010randomized}
{\sc D.~Leventhal and A.~S. Lewis}, {\em Randomized methods for linear constraints: convergence rates and conditioning}, Mathematics of Operations Research, 35 (2010), pp.~641--654.

\bibitem{ma2015convergence}
{\sc A.~Ma, D.~Needell, and A.~Ramdas}, {\em Convergence properties of the randomized extended {G}auss--{S}eidel and {K}aczmarz methods}, SIAM Journal on Matrix Analysis and Applications, 36 (2015), pp.~1590--1604.

\bibitem{natterer2001mathematics}
{\sc F.~Natterer}, {\em The Mathematics of Computerized Tomography}, SIAM, 2001.

\bibitem{Needell2009RandomizedKS}
{\sc D.~Needell}, {\em Randomized {K}aczmarz solver for noisy linear systems}, BIT Numerical Mathematics, 50 (2009), pp.~395--403.

\bibitem{scikit-learn}
{\sc F.~Pedregosa, G.~Varoquaux, A.~Gramfort, V.~Michel, B.~Thirion, O.~Grisel, M.~Blondel, P.~Prettenhofer, R.~Weiss, V.~Dubourg, J.~Vanderplas, A.~Passos, D.~Cournapeau, M.~Brucher, M.~Perrot, and E.~Duchesnay}, {\em Scikit-learn: Machine learning in {P}ython}, Journal of Machine Learning Research, 12 (2011), pp.~2825--2830.

\bibitem{petra2016single}
{\sc S.~Petra and C.~Popa}, {\em Single projection {K}aczmarz extended algorithms}, Numerical Algorithms, 73 (2016), pp.~791--806.

\bibitem{popa1999characterization}
{\sc C.~Popa}, {\em Characterization of the solutions set of inconsistent least-squares problems by an extended {K}aczmarz algorithm}, Korean Journal of Computational and Applied Mathematics, 6 (1999), pp.~51--64.

\bibitem{6868233}
{\sc O.~A. Rasheed, P.~Ivaniš, and B.~Vasić}, {\em Fault-tolerant probabilistic gradient-descent bit flipping decoder}, IEEE Communications Letters, 18 (2014), pp.~1487--1490, \url{https://doi.org/10.1109/LCOMM.2014.2344031}.

\bibitem{strohmer2009randomized}
{\sc T.~Strohmer and R.~Vershynin}, {\em A randomized {K}aczmarz algorithm with exponential convergence}, Journal of Fourier Analysis and Applications, 15 (2009), p.~262.

\bibitem{doi:10.1080/00207543.2013.850550}
{\sc D.~Verleye and E.-H. Aghezzaf}, {\em Optimising production and distribution operations in large water supply networks: A piecewise linear optimisation approach}, International Journal of Production Research, 51 (2013), pp.~7170--7189, \url{https://doi.org/10.1080/00207543.2013.850550}, \url{https://doi.org/10.1080/00207543.2013.850550}, \url{https://arxiv.org/abs/https://doi.org/10.1080/00207543.2013.850550}.

\bibitem{vyavahare2019distributed}
{\sc P.~Vyavahare, L.~Su, and N.~H. Vaidya}, {\em Distributed learning over time-varying graphs with adversarial agents}, in 2019 22th International Conference on Information Fusion (FUSION), IEEE, 2019, pp.~1--8.

\bibitem{wang2023solving}
{\sc X.~Wang, M.~Che, C.~Mo, and Y.~Wei}, {\em Solving the system of nonsingular tensor equations via randomized kaczmarz-like method}, Journal of Computational and Applied Mathematics, 421 (2023), p.~114856.

\bibitem{model_privacy}
{\sc X.~Wang, Y.~Xiang, J.~Gao, and J.~Ding}, {\em Information laundering for model privacy}, in 9th International Conference on Learning Representations, {ICLR} 2021, Virtual Event, Austria, May 3-7, 2021, OpenReview.net, 2021.

\bibitem{yang2019byrdie}
{\sc Z.~Yang and W.~U. Bajwa}, {\em {ByRDiE}: Byzantine-resilient distributed coordinate descent for decentralized learning}, IEEE Transactions on Signal and Information Processing over Networks, 5 (2019), pp.~611--627.

\bibitem{SWIRYDOWICZ2022102870}
{\sc K.~Świrydowicz, E.~Darve, W.~Jones, J.~Maack, S.~Regev, M.~A. Saunders, S.~J. Thomas, and S.~Peleš}, {\em Linear solvers for power grid optimization problems: A review of gpu-accelerated linear solvers}, Parallel Computing, 111 (2022), p.~102870, \url{https://doi.org/https://doi.org/10.1016/j.parco.2021.102870}, \url{https://www.sciencedirect.com/science/article/pii/S0167819121001125}.

\end{thebibliography}

 \appendix
\section{Single row convergence without block-list}\label{appendix:single_row}
In this section, we present an algorithm and its corresponding theory for a specific case in which only one row ($d_0=1$) is utilized for updating the solution at each iteration. Therefore, $N_r \equiv N$.  The algorithm is based on the assumption that the probability of selecting row $i \in [d_1]$ for the updating process is proportional to the squared length of the corresponding row. Further details can be found in Algorithm~\ref{alg:single_row}. We will then proceed to analyze the convergence of Algorithm~\ref{alg:single_row}.

\begin{algorithm}[h]
\caption{\textsc{Distributed Randomized Kaczmarz with block-list}}\label{alg:single_row}
\begin{algorithmic}[1]
    \STATE \textbf{Input}: Initialize \textbf{block-list $B$}, reliable worker set $D = [N]$,  $\text{MaxIter}$, $\text{Tol}$, Blocklist\_flag, checking period $T$.
    \STATE Initialize $c_s = 2 \times\text{Tol}$, a counter vector $E = 0\in \mathbb{R}^{N}$ 
    \IF{Blocklist\_flag}
      \STATE Initialize \textbf{block-list $B$}  
    \ENDIF
    \WHILE{$j < \text{MaxIter}$ and $|c_s| > \text{Tol}$, }
    \STATE The central worker $w_c$ selects a row index $i_j \in [d_1]$ with probability $p_{i_j} = \frac{\|A_{i_j}\|_2^2}{\|A\|_F^2}$
    \STATE Sample $w_1,\ldots,w_n$ uniformly from $D$
    \STATE Broadcast $A_{i_j}$ to $w_1,\ldots,w_n$ 
    \STATE  
    $w_s$ returns $c_s = \frac{\langle A_{i_j}^T,x_i\rangle-(b_{i_j}+ e_\ell(i_j))}{\|A_{i_j}\|^2} $, if  $w_s\in C_l$  
    \STATE $w_c$ splits  $\{c_s\}_{s = 1}^{n}$ into groups $G_1,\ldots,G_k$   and randomly choose from groups $G_{s}$ that satisfy $|G_{s}| \geq n(1-p)$\label{alg:line:7}

  \STATE Update $x^{j+1} = x^{j} + c_{s_{0}}A_{i_j}^\top$
  \STATE Update $E(s) = E(s) + 1,$ if $c_s\notin G_{s_0}$ 
  \label{alg:1row-line:9}
  \IF {mod$(j,T) = 0$}\label{alg:1-rowline:10}
        \STATE Update $B$ by checking the value of entries in $E$\label{alg:line:11}
        \STATE $D = D\setminus B$
  \ENDIF
  \STATE Update $j=j+1$
    \ENDWHILE
    \STATE \textbf{Output}: $x^j$ and $B$

    \end{algorithmic}
\end{algorithm}

\begin{theorem}\label{thm:main_d0=1}
Let $A\in\mathbb{R}^{d_1\times d_2}$ with $d_1\geq d_2$ and $b,e_1,\ldots,e_{k}\in\mathbb{R}^{d_1}$. Assume that we solve $Ax^*=b$ via Algorithm~\ref{alg:single_row}, then 
\begin{equation}
\mathbb{E}\|x_{i}-x^*\|_2^2
      \leq\alpha^{i+1}\|x_0-x^*\|_2^2+ \frac{1-\alpha^{i+1}}{1-\alpha}\frac{1}{\|A\|_F^2}\sum_{\ell=1}^{k}q_\ell\|e_\ell\|^2,
    \label{eqn:convg}
\end{equation}
where $\alpha=1-\frac{\sigma_{\min}^2(A)}{\|A\|_F^2}$,  $q_{\ell}=\frac{\hat{q}^{\ell}_{\text{mode}}}{q}$ and  $\sigma_{\min}(A)$ is the smallest singular value of $A$. 

Additionally, if $\|e_{\ell}\|\leq C$, we have 
\begin{equation}\label{eqn:thm1-2}
\mathbb{E}\|x_{i}-x^*\|_2^2\leq\alpha^{i+1}\|x_0-x^*\|_2^2+\frac{1-\alpha^{i+1}}{1-\alpha}\frac{Cq_0}{\|A\|_F^2}.
\end{equation}
\end{theorem}

\begin{remark}
    To provide a quantitative understanding of  Theorem \ref{thm:main_d0=1}, we present several examples in  Tables~\ref{tab:combo_k} and \ref{tab:combo_n}. For simplicity, assume that each error category has the same fraction $p_\ell = p/k$. Thus, all $\hat{q}_{\text{mode}}^{\ell}$ are equal. Here $q_0$ is the probability that the algorithm chooses the right mode and $q$ is the probability that there is a mode.  In these two tables, we present the values for $\hat{q}_{mode}^{\ell},\hat{q}_{mode}^{0},q$ and $q_0$ by varying the number of error categories $k$, the number of chosen workers $n$ and the adversarial rate $p$. These two tables are generated by solving a linear system with a row-normalized matrix $A\in\mathbb{R}^{1000\times 100}$.

 As $k$ increases, $q_\ell$ decreases and $q_0$ increases. Therefore, the error bound in  \eqref{eqn:thm1-2} decreases with respect to $k$ and thus reaches better convergence. When $k$ is large enough, $ q_{\ell}\approx0$. Therefore, when the noise is uniformly random error and there is a mode for the step-size, the mode will be the correct mode.
 As $n$ increases, there is a similar decrease effect and therefore a better convergence.
 \begin{table}[h]
\caption{Total number of workers $N=100$, number of chosen workers $n=5$.}
    \label{tab:combo_k}
    \centering
    \begin{tabular}{||c|c|c|c|c|c||}
         \hline
         $p$ &$k$ & $\hat{q}_{mode}^{\ell}$ & $\hat{q}_{mode}^{0} $  & $q$  &$q_0$\\
         \hline
         \multirow{3}{4em}{\quad $0.8$}
         &$5$ &$0.1$ &$0.16$ &$0.67$  &$0.15$\\
         &$10$ & $0.04$ &$0.21$ &$0.57$  &$0.36$\\
        &$15$ &$0.02$ & $0.23$ & $0.48$  &$0.46$\\
         \hline
        \multirow{3}{4em}{\quad$0.2$} 
        &$3$ &$0.002$ &$0.63$ &$0.64$  &$0.98$\\
        &$5$ &$8\times10^{-4}$ &$0.65$ &$0.65$  &$0.99$\\
         &$10$ &$2\times10^{-4}$ &$0.66$ &$0.67$  &$0.99$\\
         &$15$ &$2\times10^{-4}$ &$0.685$ &$0.689$ &$0.99$\\
         \hline
    \end{tabular}
\end{table}
 
\begin{table}[h]
\caption{Total number of workers $N=100$, number of error categories $k=5$.}
    \label{tab:combo_n}
    \centering
    \begin{tabular}{||c|c|c|c|c|c||}
         \hline
         $p$ &$n$ & $\hat{q}_{mode}^{\ell} $ & $\hat{q}_{mode}^{0}$  & $q$ &$q_0$\\
         \hline
         \multirow{3}{4em}{\quad $0.8$}
         & 10 &0.099 &0.18 &0.67 &0.26\\
         &15 & 0.099 &0.2 &0.7  &0.29\\
         & 20 &0.097 &0.23 &0.71  &0.31\\
         \hline
        \multirow{3}{4em}{\quad$0.2$} 
        &10 &$7\times10^{-6}$ &0.904 &0.90 &$1 - 5\times10^{-6}$ \\
        &15 &$5\times10^{-7}$ &0.97 &0.97 &$1- 3\times10^{-6}$\\
        &20 &$1\times10^{-7}$ &0.99  &0.99 &$1- 6\times10^{-7}$\\
         \hline
    \end{tabular}
\end{table}
\end{remark}
\begin{proof}[Proof of Theorem \ref{thm:main_d0=1}]\label{appendix:pf_d0=1}
To prove   \eqref{eqn:convg}, at each iteration, we consider solving $Ax=b$, $Ax=b+e_1$,$\ldots$, $Ax=b+e_k$ with probability $q_0,q_1,\cdots,q_k$, respectively. Therefore, for the $(i+1)$-th step, we have the iteration 
$$x_{i+1}=x_{i}-\frac{\langle A_j,x_i\rangle-b_j}{\|A_j\|^2}A_j^{\top},$$
or
$$x_{i+1}=x_{i}-\frac{\langle A_j,x_i\rangle-(b_j+e_\ell(j))}{\|A_j\|^2}A_j^{\top}. $$ for $\ell=1,\cdots,k$, $A_j$ is the $j$-th row of matrix $A$.

Notice that when $x_{i+1}=x_{i}-\frac{\langle A_j,x_i\rangle-b_j}{\|A_j\|^2}(A_j)^{\top}$, we have
\begin{equation}\label{eqn:noise_free}
    \begin{aligned}
      &\mathbb{E}_{j}\|x_{i+1}-x^*\|_2^2\\
      =&   \mathbb{E}_{j}\|x_{i}-\frac{\langle A_j,x_i\rangle-b_j}{\|A_j\|^2}(A_j)^{\top}-x^*\|_2^2\\
      \leq &\left(1-\frac{\sigma_{\min}^2(A)}{\|A\|_F^2} \right)\|x_i-x^*\|_2^2.
    \end{aligned}
\end{equation}
When $x_{i+1}=x_{i}-\frac{\langle A_j,x_i\rangle-(b_j+e_\ell(j))}{\|A_j\|^2}(A_j)^{\top}$, we have
\begin{equation}\label{eqn:noise_ell}
    \begin{aligned}
      &\mathbb{E}_{j}\|x_{i+1}-x^*\|_2^2\\
      =&   \mathbb{E}_{j}\|x_{i}-\frac{\langle A_j,x_i\rangle-(b_j+e_\ell(j))}{\|A_j\|_2^2}A^{\top}(j,:)-x^*\|_2^2\\
      \leq &\left(1-\frac{\sigma_{\min}^2(A)}{\|A\|_F^2} \right)\|x_i-x^*\|_2^2 +\mathbb{E}_{j}\frac{e_{\ell}^2(j)}{\|A_j\|_2^2}\\
      =&\left(1-\frac{\sigma_{\min}^2(A)}{\|A\|_F^2} \right)\|x_i-x^*\|_2^2 + \frac{\|e_{\ell}\|^2}{\|A\|_F^2}.
    \end{aligned}
\end{equation}
Combining \eqref{eqn:noise_free} and \eqref{eqn:noise_ell}, we have
\begin{equation}
    \begin{aligned}
      & \mathbb{E}_{j}\|x_{i+1}-x^*\|_2^2\\
      \leq&\left(1-\frac{\sigma_{\min}^2(A)}{\|A\|_F^2} \right)\|x_i-x^*\|_2^2+\frac{1}{\|A\|_F^2}\sum_{\ell=1}^{k}q_{\ell}\|e_\ell\|^2.
    \end{aligned}
\end{equation}
Set $\alpha=1-\frac{\sigma_{\min}^2(A)}{\|A\|_F^2}$.
Therefore, 
\begin{equation}
    \begin{aligned}
      & \mathbb{E}\|x_{i+1}-x^*\|_2^2\\
      \leq&\alpha^{i+1}\|x_0-x^*\|_2^2+ \frac{1-\alpha^{i+1}}{1-\alpha}\frac{1}{\|A\|_F^2}\sum_{\ell=1}^{k}q_\ell\|e_\ell\|^2.
    \end{aligned}
\end{equation}
\end{proof}
\section{Discussion of the optimal number of used rows}
\label{appendix:opt_d0}
Assume that   row $r$ is distributed to $N_r$ workers and $n_r$ workers are selected to involve in the computation of row $r$ with   $N_r\equiv N, n_r \equiv n$ for all $r$. Additionally, we assume that the probability of each error category $\ell$ ($\ell \neq 0$) for each row is the same with $p_{r,\ell}\equiv p/k.$ Moreover, we have the restricted minimal mode number $g_0(r) \equiv g_0 = \max(\lceil\frac{n}{k+1}\rceil, \lceil n(1-p)\rceil)$, and 
\begin{equation}
    \begin{aligned}
      &b^r_g \equiv b_g, \text{coefficient of term } x^n \text{ of} \left( \sum_{j=0}^{g-1}\binom{Np/k}{j}x^j\right)^{k-1}\left( \sum_{j=0}^{g-1}\binom{N(1-p)}{j}x^j\right),\\
      &a^r_{g,\ell} \equiv a_g, \text{coefficient of term } x^{n-g} \text{ of} \left( \sum_{j=0}^{g-1}\binom{Np/k}{j}x^j\right)^{k-1},\\
      &q^r_g \equiv q_g = \frac{\binom{Np/k}{g}a_g}{\binom{N}{n}},\\
      & Q(r,\tau_i) \equiv \sum_{g = g_0}^{n}q_g\left({b_g}/{\binom{N}{n}}\right)^{d_0-1}=Q_{\max}=Q_{\min}\coloneqq Q,\\
      &\beta_r = \frac{d_0}{d_1}Q_{\max}.
    \end{aligned}
\end{equation}
Let $\alpha(d_0) = \alpha = 1 - Q\frac{d_0}{d_1}\sigma_{\min}^2(\tilde{A})$. To study the relation between $d_0$ and the convergence, consider
\begin{equation}
    \begin{aligned}
      \frac{\partial \alpha(d_0)}{\partial d_0} \propto - \sum_{g = g_0}^{n}q_g\left(1+d_0\log\left({b_g}/{\binom{N}{n}}\right)\right)\left({b_g}/{\binom{N}{n}}\right)^{d_0-1}.
    \end{aligned}
\end{equation}
If $d_0 \geq - \frac{1}{\log(b_g/\binom{N}{n})} $ for all $g$, then $\frac{\partial \alpha(d_0)}{\partial d_0} \geq 0$. This implies that as $d_0$ increases, $\alpha(d_0)$ increases.
When $g_0 = n$, we have
$$\frac{\partial \alpha(d_0)}{\partial d_0} \propto - \left(1+d_0\log\left({b_n}\big/{\binom{N}{n}}\right)\right)\left({b_n}\big/{\binom{N}{n}}\right)^{d_0-1},$$ 
and to reach the fastest convergence rate, $d_0 = - \frac{1}{\log\left( {b_g}/{\binom{N}{n}}\right)}$.
One can explore the minimizers for $\alpha$ in more general cases, where multiple local minimizers could present in the landscape.   

\section{Proof of Lemma~\ref{lem:mode g}}\label{pf:lem_mode_g}
\begin{proof}[Proof of Lemma~\ref{lem:mode g}]

Given a row $r$, the number of combinations where the mode belongs to category $\ell$ with mode count $g$ can be divided into two parts: the combinations of workers in category $\ell$ and the combinations of workers in all other categories excluding $\ell$.

We start by calculating the number of combinations of workers in the remaining categories. This is equivalent to selecting $(n_r - g)$ balls from $k$ bins subject to constraints that the $\tilde{\ell}$-th bin contains $N_rp_{r,\tilde{\ell}}$ balls, with a maximum of $g$ balls that can be chosen from this bin for all $\tilde{\ell} \neq \ell$. With these constraints, there are $\binom{N_rp_{r,\tilde{\ell}}}{j}$ ways to choose $j$ balls from bin $\tilde{\ell}$. By letting $j$ vary and considering the $k$ bins, the total number of valid combinations is equal to the coefficient of the term $x^{n_r-g}$ in the polynomial  $\prod\limits_{\tilde{\ell}=0,\tilde{\ell}\neq \ell}^{k}\sum_{j=0}^{g-1}\binom{N_rp_{r,\tilde{\ell}}}{j}x^j$.
The number of combinations of workers in category $\ell$ with mode count $g$ is $\binom{N_rp_{r,\ell}}{g}$. Finally, the total number of combinations is given by $\binom{N_r}{n_r}$. This concludes the proof of Lemma~\ref{lem:mode g}.
\end{proof}

\end{document}